\documentclass[11pt, reqno]{amsart}

\usepackage{amsthm,amssymb,amstext,amscd,amsfonts,amsbsy,amsrefs,amsxtra,latexsym,amsmath,xcolor,mathrsfs,fancybox,upgreek, soul,url}
\usepackage[english]{babel}
\usepackage[all,cmtip]{xy}
\usepackage[latin1]{inputenc}
\usepackage{cancel}
\usepackage[draft]{hyperref}
\usepackage{comment}
\usepackage{mdframed}
\allowdisplaybreaks
\usepackage{mathtools}
\usepackage{enumerate}

\newtheorem{theorem}{Theorem}
\newtheorem{lemma}[theorem]{Lemma}

\newtheorem{proposition}[theorem]{Proposition}

\theoremstyle{definition}
\newtheorem{definition}{Definition}

\theoremstyle{remark}

\newtheorem*{remark}{Remark}
\newtheorem*{example}{Example}
\numberwithin{theorem}{section}
\numberwithin{equation}{section}

\newcommand{\N}{\mathbb{N}}
\newcommand{\Z}{\mathbb{Z}}
\newcommand{\R}{\mathbb{R}}
\newcommand{\C}{\mathbb{C}}

\newcommand{\Q}{\mathbb{Q}}

\newcommand{\sgn}{\operatorname{sgn}}

\newcommand{\J}{\mathscr{S}}
\def\H{\mathbb{H}}
\renewcommand{\pmod}[1]{\  \,  \left( \mathrm{mod} \,  #1 \right)}

\begin{document}

\title[A family of vector-valued quantum modular forms of depth two ]{A family of vector-valued quantum modular forms of depth two}

\author{Joshua Males}

\address{Mathematical Institute, University of Cologne, Weyertal 86-90, 50931 Cologne, Germany}
\email{jmales@math.uni-koeln.de}

\begin{abstract}We introduce and investigate an infinite family of functions which are shown to have generalised quantum modular properties. We realise their ``companions" in the lower half plane both as double Eichler integrals and as non-holomorphic theta functions with coefficients given by double error functions. Further, we view these Eichler integrals in a modular setting as parts of certain weight two indefinite theta series.
\end{abstract}

\maketitle

\section{Introduction and statement of results}

In a celebrated paper of Zagier, the concept of quantum modular forms is introduced, following investigations into Kontsevich's ``strange" function \cite{zagier2010quantum,zagier2001vassiliev}, given by
\begin{equation*}
K(q) \coloneqq 1 + \sum_{n = 1}^{\infty} (q;q)_n,
\end{equation*}
where $(a;q)_n \coloneqq \prod_{j = 0}^{n-1} (1-aq^j)$ for $n \in \N_0 \cup \{ \infty \}$ is the $q$-Pochhammer symbol, and $q \coloneqq e^{2 \pi i \tau}$ with $\tau \in \mathbb{H}$. In particular, $K(q)$ does not converge on any open subset of $\C$, but is seen to be a finite sum at any root of unity. Zagier shows that at roots of unity $\zeta$, the function $K(\zeta)$ agrees to infinite order with the Eichler integral of $\eta(\tau) \coloneqq q^{\frac{1}{24}} (q;q)_\infty$ (see page 959 of \cite{zagier2001vassiliev} for the precise definition of the Eichler integral in this context), and hence inherits the Eichler integral's quantum modular properties.

Here we give a brief description of the essence of what a quantum modular form is, and for a full introduction refer the reader to e.g. Chapter 21 of \cite{bringmann2017harmonic}. A quantum modular form is essentially a function $f \colon \mathcal{Q} \rightarrow \C$ for some fixed $ \mathcal{Q} \subseteq \mathbb{Q}$, whose errors of modularity (for $M = \left( \begin{smallmatrix} a & b \\ c & d \end{smallmatrix} \right) \in \text{SL}_2 (\Z)$)

\begin{equation}\label{Equation: error of modularity}
f(\tau) - (c \tau + d)^k f (M \tau)
\end{equation}
are in some sense ``nicer" than the original function. Often, for example, the original function $f$ is defined only on $\Q$, but the errors of modularity can be defined on some open subset of $\R$. The set $\mathcal{Q}$ is called the quantum set of the function $f$.  One may also consider quantum modular forms for $M \in \Gamma$, a subgroup of $\text{SL}_2 (\Z)$.  Further, Zagier also considered so-called ``strong" quantum modular forms, where one considers asymptotic expansions and not just values. Leaving this definition of quantum modular forms intentionally vague allowed Zagier to collect many examples in the same heading.

 Since their introduction, there has been an explosion of research into quantum modular forms in many guises, and they appear in work in many areas. For example, in \cite{folsom2017strange} the authors  consider a certain generalisation of $K(q)$ and investigate its quantum properties. It is shown to have intricate connections to the Habiro ring (introduced in \cite{habiro2004cyclotomic}) and implications therein to combinatorics, in particular to the generating function for ranks of strongly unimodal sequences, are explored. 

There are also deep connections between quantum modular forms and other areas. For example, the connection between them and mock modular forms (surveyed in e.g. \cite{ono2009unearthing}) is investigated in papers such as \cite{bringmann2015unimodal,bringmann2016half,bryson2012unimodal}, among others.
Furthermore, interesting examples of quantum modular forms exist in the interface of physics and knot theory, see e.g. a study of Kashaev invariants of $(p,q)$-torus knots in \cite{hikami2003torus,hikami2015torus} and investigations of Zagier into limits of quantum invariants of $3$-manifolds and knots \cite{zagier2010quantum} - indeed, this is the reason that Zagier chose the name ``quantum" modular forms.

An additional example is given in \cite{bringmann2015,creutzig2016regularised}, where characters of vertex operator algebras are explored, and it is shown that natural parts of these characters are quantum modular forms (of depth one). Motivated in part by these discoveries, the authors of \cite{Higher_depth_QMFs} consider higher-dimensional analogues, defining so-called higher depth quantum modular forms, and provide two examples of such forms of depth two. In the simplest case, these are functions that satisfy

\begin{equation*}
f(\tau) - (c \tau + d)^k f(M \tau) \in \mathsf{Q}_k (\Gamma) \mathcal{O} (R) + \mathcal{O} (R),
\end{equation*}
where $\mathsf{Q}_k (\Gamma)$ is the space of quantum modular forms of weight $k$ on $\Gamma$, and $\mathcal{O}(R)$ is the space of real-analytic functions on $R \subset \R$. As noted in \cite{Higher_depth_QMFs}, the easiest (trivial) examples come from multiplying two depth one forms - however, the two examples discussed therein appear to be non-trivial examples. 

Again, these examples arise from a physics perspective. In fact, they come from the character of a vertex operator algebra $W(p)_{A_2}$, where $p \geq 2$, associated to the root lattice of type $A_2$ of the simple Lie algebra $\mathfrak{sl}_3$. The authors show that the character can be decomposed into two distinct functions, each of which are quantum modular forms of depth two on some subgroup of the full modular group. In a follow-up paper \cite{Higher_depth_QMFs_2} the same authors also show that their functions can be viewed as vector-valued quantum modular forms of depth two on all $\text{SL}_2 (\Z)$ (see Section \ref{Section: defining QMFs} for definitions).

In this paper we require \eqref{Equation: error of modularity} to be real-analytic, and the functions we consider will satisfy the properties of strong quantum modular forms. We construct a generalisation of a function called $F_1$ defined in \cite{Higher_depth_QMFs,Higher_depth_QMFs_2}. In doing so, we provide an infinite family of non-trivial vector-valued quantum modular forms of depth two. We define our generalisation $F$ as a sum of three terms, $F(q) \coloneqq F_1 (q) + F_2(q) + F_3 (q)$ where  

\begin{equation*}
F_1 (q) \coloneqq \sum_{\alpha \in \J} \varepsilon(\alpha) \sum_{n \in \alpha + \N_0^2 } q^{Q(n)}
\end{equation*}
is a weighted sum of partial theta functions, and where $F_2, F_3$ are one-dimensional sums arising from the boundary term $n = 0$ in a double Eichler integral. Here, $Q(n)$ is a positive definite integral binary quadratic form, $\J$ is a finite set of pairs in $\mathbb{Q}^2$, and $\varepsilon \colon \J \rightarrow \R \backslash \{0\}$. Both $\J$ and $\varepsilon$ are required to satisfy some symmetry conditions - see Section \ref{Section: F, J and Q(x)} for the full definitions.

\begin{remark}
	The function $F_1$ of Bringmann, Kaszian, and Milas as defined in \cite{Higher_depth_QMFs,Higher_depth_QMFs_2} is a direct specialization of the function $F$ presented here, specialized to a certain set of six pairs of rational points, the specific quadratic form $Q(x) = 3x_1^2 + 3x_1x_2 + x_2^2$, and a fixed $\varepsilon$.
	In particular we have conflicting notation - note that the functions $F_1,F_2$ given in \cite{Higher_depth_QMFs,Higher_depth_QMFs_2} and the functions $F_1,F_2$ given in the present paper are different. 
\end{remark}

Analagously to \cite{Higher_depth_QMFs_2}, we show that $F$ satisfies the following (see Theorem \ref{Theorem: Main Theorem - Paper} for a precise statement).

\begin{theorem}\label{Theorem: introduction version of main theorem}
	The function $F$ is a sum of components of a vector-valued quantum modular form of depth two and weight one on $\text{SL}_2 (\Z)$ with some explicit quantum set $\mathcal{Q}$ defined in Section \ref{Section: quantum set}. In some special cases, $F$ itself is a single component of a vector-valued form.
\end{theorem}

Though here we only show the vector-valued version, we note that it is also possible to show that our function $F$ is a quantum modular form of depth two and weight one itself, on a suitably chosen congruence subgroup of $\text{SL}_2 (\Z)$, generalising the situation in \cite{Higher_depth_QMFs}. The connection for Theorem \ref{Theorem: introduction version of main theorem} is made by relating $F$ asymptotically at certain roots of unity to a double Eichler integral $\mathcal{E}$ of the shape

\begin{equation*}
\int_{- \bar{\tau}}^{i \infty} \int_{\omega_1}^{i \infty} \frac{ f_1(\omega_1) f_2(\omega_2) }{ \sqrt{-i(\omega_1 + \tau)} \sqrt{-i(\omega_2 + \tau)}} d \omega_2 d \omega_1 ,
\end{equation*}
where the $f_j $ lie in the space of vector-valued modular forms on $\text{SL}_2 (\Z)$. By a result of \cite{Higher_depth_QMFs_2}, such Eichler integrals possess higher depth vector-valued quantum modular properties (see Proposition \ref{Proposition: transformation for vector-valued}), and so by virtue of the asymptotic agreement at points in $\mathcal{Q}$ of $F$ and $\mathcal{E} $, the function $F$ inherits these properties.

We then place the Eichler integral $\mathcal{E} $ into a modular setting by relating it to an indefinite theta function (see Proposition \ref{Proposition: indefinite theta function} for a precise statement). 

\begin{proposition}
	The indefinite theta function of signature $(2,2)$ defined in Section \ref{Section: Completed Indefinite theta functions 2} has purely non-holomorphic part $\Theta(\tau) \mathcal{E} (\tau)$, where $\Theta(\tau)$ is a theta series of signature $(2,0)$. 
\end{proposition}

The paper is organised as follows. We begin in Section \ref{Section: Prelims} by reviewing basic properties of special functions, and detailing results that will be needed throughout the paper. In Section \ref{Section: F, J and Q(x)} we introduce the function $F$ that we concentrate on for the rest of the paper. We define the quantum set $\mathcal{Q}$ in Section \ref{Section: quantum set} before we find the asymptotic behaviour of  $F$ at certain roots of unity in Section \ref{Section: asymptotic behaviour of F at Certain Roots of Unity}. In Section \ref{Section: Multiple Eichler Integrals of weight 1} a double Eichler integral is introduced and shown, via the use of Shimura theta functions, to exhibit modular properties. Next we turn to Section \ref{Section: Indefinite theta functions 1} where we show that the double Eichler integral can be viewed as a piece of a certain indefinite theta series. Given results in this section, we proceed to prove the main results regarding quantum modularity of $F$ in Section \ref{Section: Proof of main theorem}. We set the double Eichler integral in a modular setting in Section \ref{Section: Completed Indefinite theta functions 2}, using boosted complementary error functions and a result of \cite{Generalised_error_functions}. Finally, we conclude the paper in Section \ref{Section: further questions} with some questions which will be investigated in further work.

\section*{Ackowledgments}

The author would like to thank Kathrin Bringmann, Stephan Ehlen, Jonas Kaszian, and Larry Rolen for insightful discussions as well as useful comments on an earlier version of this paper. The author would also like to thank the referee for numerous helpful comments.

\section{Preliminaries}\label{Section: Prelims}
We begin by introducing some basic functions along with recalling relevant results pertinent to the rest of the paper.
\subsection{Error functions}
We first define a rescaled version of the usual one-dimensional error function. For $u \in \R$ set

\begin{equation}
E(u) \coloneqq 2 \int_0^u e^{- \pi \omega^2} d \omega .
\end{equation}
This has first derivative given by $E'(u) = 2 e^{- \pi u^2}$. The function $E(u)$ may also be written using incomplete gamma functions $\Gamma(a, u) \coloneqq \int_{u}^{\infty} e^{- \omega} \omega^{a-1} d \omega$, with $a > 0$, via the formula

\begin{equation}\label{Equation: E(u) in terms of incomplete gamma functions}
E(u) = \sgn(u) \left( 1 - \frac{1}{\sqrt{\pi}} \Gamma\left(\frac{1}{2} , \pi u^2\right)  \right) ,
\end{equation}
where we set 

\begin{equation*}
\sgn(x) \coloneqq \begin{dcases}
1 & \text { if } x > 0 ,\\
0 & \text { if } x = 0 ,\\
-1 & \text { if } x < 0 .\\
\end{dcases}
\end{equation*}
We will also make use of an augmented $\sgn$ function, defined by $\sgn^*(x) \coloneqq \sgn(x)$ for $x \neq 0$ and $\sgn^* (0) \coloneqq 1$.

We also require, for non-zero $u$, the function

\begin{equation*}
M(u) \coloneqq \frac{i}{\pi} \int_{\R - iu} \frac{e^{ - \pi \omega^2 - 2 \pi i u \omega}}{\omega} d \omega.
\end{equation*}
A relation between $M(u)$ and $E(u)$, for non-zero $u$, is given by

\begin{equation}
M(u) = E(u) - \sgn(u).
\end{equation}
Therefore, using \eqref{Equation: E(u) in terms of incomplete gamma functions}, we have that

\begin{equation}\label{Equation: M(u) in terms of incomplete gamma functions}
M(u) =\frac{ - \sgn(u) }{\sqrt{\pi}} \Gamma\left(\frac{1}{2} , \pi u^2\right) .
\end{equation}

We further need the two-dimensional analogues of the above functions. Following \cite{Generalised_error_functions} and changing notation slightly, we define $E_2 \colon \R \times \R^2 \rightarrow \R$ by

\begin{equation*}
E_2 (\kappa ; u) \coloneqq \int_{\R^2} \sgn(\omega_1) \sgn(\omega_2 + \kappa \omega_1) e^{- \pi \left( (\omega_1 - u_1)^2 + (\omega_2 - u_2)^2 \right)} d \omega_1 d \omega_2, 
\end{equation*}
where throughout we denote components of vectors just with subscripts. Note that

\begin{equation*}
E_2 (\kappa; -u) = E_2 (\kappa; u).
\end{equation*}
Again following \cite{Generalised_error_functions}, for $u_2, u_1 - \kappa u_2 \neq 0$, we define

\begin{equation*}
M_2 (\kappa; u_1, u_2) \coloneqq - \frac{1}{\pi^2} \int_{\R - iu_2} \int_{\R - i u_1} \frac{ e^{ - \pi \omega_1^2 - \pi \omega_2^2 - 2 \pi i (u_1 \omega_1 + u_2 \omega_2)}}{\omega_2 (\omega_1 - \kappa \omega_2)} d\omega_1 d \omega_2 .
\end{equation*}
Then we have that

\begin{equation}\label{Equation: relation between M_2 and E_2}
\begin{split}
M_2 (\kappa; u_1 , u_2) = & E_2(\kappa; u_1 , u_2) - \sgn(u_2) M (u_1) \\
& - \sgn(u_1 - \kappa u_2) M \left( \frac{u_2 + \kappa u_1}{\sqrt{1 + \kappa^2}} \right) - \sgn(u_1) \sgn(u_2 + \kappa u_1).
\end{split}
\end{equation}
The relation \eqref{Equation: relation between M_2 and E_2} extends the definition of $M_2 (u)$ to include $u_2 = 0$ or $u_1 = \kappa u_2$ - note however that $M_2$ is discontinuous across these loci. Putting $x_1 \coloneqq u_1 - \kappa u_2 , x_2 \coloneqq u_2$ yields 

\begin{equation}
\begin{split}
M_2 (\kappa; u_1 , u_2) = & E_2(\kappa; x_1 + \kappa x_2 , x_2) + \sgn(x_1)\sgn(x_2) \\
& - \sgn(x_2) E(x_1 + \kappa x_2) - \sgn(x_1) E\left( \frac{ \kappa x_1}{\sqrt{1 + \kappa^2}} + \sqrt{1+ \kappa^2} x_2 \right).
\end{split}
\end{equation}
We also have the first partial derivatives of $M_2$ as 

\begin{equation*}
\begin{split}
& M_2^{(1,0)}(\kappa; u_1, u_2) = 2 e^{-\pi u_1^2} M(u_2) + \frac{2 \kappa}{\sqrt{1 + \kappa^2}} e^{ \frac{- \pi (u_2 + \kappa u_1)^2 }{1 + \kappa^2}} M \left(\frac{u_1 - \kappa u_2}{\sqrt{1+ \kappa^2}}\right) ,\\
& M_2^{(0,1)} (\kappa; u_1, u_2) = \frac{2}{\sqrt{1 + \kappa^2}} e^{ \frac{- \pi (u_2 + \kappa u_1)^2 }{1 + \kappa^2}} M \left(\frac{u_1 - \kappa u_2}{\sqrt{1+ \kappa^2}}\right) ,
\end{split}
\end{equation*}
along with the first partial derivatives of $E_2$

\begin{equation*}
\begin{split}
& E_2^{(1,0)}(\kappa; u_1, u_2) = 2 e^{-\pi u_1^2} E(u_2) + \frac{2 \kappa}{\sqrt{1 + \kappa^2}} e^{ \frac{- \pi (u_2 + \kappa u_1)^2 }{1 + \kappa^2}} E \left(\frac{u_1 - \kappa u_2}{\sqrt{1+ \kappa^2}}\right) ,\\
& E_2^{(0,1)} (\kappa; u_1, u_2) = \frac{2}{\sqrt{1 + \kappa^2}} e^{ \frac{- \pi (u_2 + \kappa u_1)^2 }{1 + \kappa^2}} E \left(\frac{u_1 - \kappa u_2}{\sqrt{1+ \kappa^2}}\right) ,
\end{split}
\end{equation*}
all of which follow from Proposition 3.3. of \cite{Generalised_error_functions}.

\subsection{Euler-Maclaurin summation formula}
We state two special cases of the Euler-Maclaurin summation formula, in one and two dimensions, as needed for this paper. \\
Let $B_m (x)$ be the $m$th Bernoulli polynomial which is defined by $\frac{t e^{xt}}{e^t - 1} \eqqcolon \sum_{m \geq 0} B_m (x) \frac{t^m}{m!}$. We recall the property

\begin{equation}\label{equation: Bernoulli polynomial transformation}
B_m (1-x) = (-1)^m B(x).
\end{equation}

The one dimensional case follows a result of Zagier in \cite{zagiervaleurs}, and it implies that, for $\alpha \in \R$ and $F \colon \R \rightarrow \R$ a $C^\infty$ function of rapid decay,
\begin{equation}\label{Euler-Maclaurin summation formula one dim}
\begin{split}
\sum_{n \in \N_0}  F ((n + \alpha) t) \sim &  \frac{ \mathcal{I}_F }{t} - \sum_{ n \geq 0} \frac{B_{n + 1} (\alpha)}{(n + 1)!} F^{(n)} (0) t^{n},
\end{split}
\end{equation}
where we set $\mathcal{I}_F = \int_0^\infty F(x) dx $. By $\sim$ we mean that the difference between the left- and right-hand side is $O(t^N)$ for any $N \in \N$.

We now turn to the two-dimensional case. Let $\alpha \in \R^2$ and $F \colon \R^2 \rightarrow \R$ a $C^\infty$ function of rapid decay. The Euler-Maclaurin summation formula in two dimensions then implies that (generalising another result of Zagier in \cite{zagiervaleurs} to include shifts by $\alpha$)

\begin{equation}\label{Euler-Maclaurin summation formula}
\begin{split}
\sum_{n \in \N_0^2}  F ((n + \alpha) t) \sim &  \frac{ \mathcal{I}_F }{t^2} - \sum_{n_2 \geq 0 } \frac{B_{n_2 + 1} (\alpha_2)}{(n_2 + 1)!} \int_0^\infty F^{(0, n_2)} (x_1 ,0) dx_1 t^{n_2 - 1}  \\
& - \sum_{n_1 \geq 0 } \frac{B_{n_1 + 1} (\alpha_1)}{(n_1 + 1)!} \int_0^\infty F^{(n_1, 0)} ( 0, x_2) dx_2 t^{n_1 - 1} \\
& + \sum_{ n_1 , n_2 \geq 0} \frac{B_{n_2 + 1} (\alpha_2)}{(n_2 + 1)!} \frac{B_{n_1 + 1} (\alpha_1)}{(n_1 + 1)!} F^{(n_1, n_2)} (0,0) t^{n_1 + n_2},
\end{split}
\end{equation}
here with $\mathcal{I}_F = \int_0^\infty \int_0^\infty F(x_1, x_2) dx_1 dx_2 $.

\subsection{Shimura's theta functions}\label{Section: Shimura theta functions}
In \cite{shimura1973modular} Shimura gave transformation laws of certain theta series, which we require here. For $\nu \in \{ 0,1\}, h \in \Z$ and $ N,A \in \N$ with $A \mid N , N \mid hA$ define

\begin{equation}\label{Shimura theta function}
\Theta_\nu (A, h, N; \tau) \coloneqq \sum_{\substack{ m \in \Z \\ m \equiv h \pmod{N}}} m^\nu q^{ \frac{A m^2}{2 N^2}} ,
\end{equation}
where $\tau \in \mathbb{H}$ and $q \coloneqq e^{2 \pi i \tau}$, as usual. Then we have the following transformation formula

\begin{equation}\label{Shmiura theta transformation formula}
\Theta_\nu (A, h, N; M \tau) = e \left( \frac{ab A h^2}{2 N^2} \right) \left( \frac{2 A c}{d} \right) \varepsilon_d (c \tau + d)^{\frac{1}{2} + \nu} \Theta_\nu (A, ah, N; \tau),
\end{equation}
for $M = ( \begin{smallmatrix} a & b \\ c & d \end{smallmatrix} ) \in \Gamma_0 (2 N)$ with $2 \mid b$. Here $e(x)  \coloneqq e^{2 \pi i x}$ and, for odd $d$, $\varepsilon_d = 1$ or $i$ depending on whether $d \equiv 1 \pmod{4}$ or $d \equiv 3 \pmod{4}$ respectively, and $\left( \frac{c}{d} \right)$ is the extended Jacobi symbol. Further, we have that

\begin{equation}\label{Equation: transformation of Shimura for vector valued}
\Theta_\nu \left(A, h, N ; -\frac{1}{\tau} \right) = (- i)^\nu (-i \tau)^{\frac{1}{2} + \nu} A^{- \frac{1}{2}} \sum_{\substack{k \pmod{N} \\ Ak \equiv 0 \pmod{N}}} e\left( \frac{Akh}{N^2} \right) \Theta_\nu (A,k,N; \tau)  .
\end{equation}
We also require the transformations

\begin{equation*}\begin{split}
& \Theta_\nu (A, -h, N; \tau) = (-1)^\nu \Theta_\nu (A, h, N; \tau),
\end{split}
\end{equation*}
and if $h_1 \equiv h_2 \pmod{N}$, then

\begin{equation*}
\Theta_\nu (A, h_1, N; \tau) = \Theta_\nu (A, h_2, N; \tau) .
\end{equation*}

\subsection{Vector-valued quantum modular forms}\label{Section: defining QMFs}
Since the study of vector-valued quantum modular forms has been motivated in the introduction, here we give only the formal definition, following \cite{Higher_depth_QMFs_2}. We begin with the depth one case, before defining those of higher depth.

\begin{definition}
	For $1 \leq j \leq N \in \N$, a collection of functions $f_j \colon \mathcal{Q} \rightarrow \C$ is called a vector-valued quantum modular form of weight $k$ and multiplier $\chi = (\chi_{j,\ell})_{1 \leq j,\ell \leq N}$ and quantum set $\mathcal{Q}$ for $\text{SL}_2 (\Z)$ if, for all $M = \left( \begin{smallmatrix} a & b \\ c & d \end{smallmatrix} \right) \in \text{SL}_2 (\Z)$ we have that
	
	\begin{equation*}
	f_j (\tau) - (c \tau + d)^{-k} \sum_{1 \leq \ell \leq N} \chi_{\ell,j}^{-1} (M) f_\ell(M \tau)
	\end{equation*}
	can be extended to an open subset of $\R$ and is real-analytic there. We denote the vector space of these forms by $\mathsf{Q}_{k} (\chi)$.
	
\end{definition}

\subsection{Higher depth vector-valued quantum modular forms}
We now consider generalisations of vector-valued quantum modular forms, again following \cite{Higher_depth_QMFs_2}.

\begin{definition}
	For $1 \leq j \leq N \in \N$, a collection of functions $f_j \colon \mathcal{Q} \rightarrow \C$ is called a vector-valued quantum modular form of depth $P$,  weight $k$ and multiplier $\chi = (\chi_{j,\ell})_{1 \leq j,\ell \leq N}$ and quantum set $\mathcal{Q}$ for $\Gamma$ if, for all $M = \left( \begin{smallmatrix} a & b \\ c & d \end{smallmatrix} \right) \in \text{SL}_2 (\Z)$ we have that
	
	\begin{equation*}
	f_j (\tau) - (c \tau + d)^{-k} \sum_{1 \leq \ell \leq N} \chi_{\ell,j}^{-1} (M) f_\ell(M \tau) \in \bigoplus_{\ell} \mathsf{Q}_{\kappa_j}^{P_\ell} (\chi_\ell) \mathcal{O} (R),
	\end{equation*}
	where $\ell$ runs through a finite set, $\kappa_\ell \in \frac{1}{2} \Z$, $P_\ell \in \Z$ with max$(P_\ell) = P-1$, $\chi_l$ multipliers, $\mathcal{O}_R$ is the space of real analytic functions on $R \subset \R$ which contains an open subset of $\R$. We also define $\mathsf{Q}_{k}^{1} (\chi) \coloneqq \mathsf{Q}_{\kappa} ( \chi)$, $\mathsf{Q}_{k}^{0} (\chi) \coloneqq 1$, and let $\mathsf{Q}_{k}^{P} (\chi)$ denote the space of forms of weight $k$, depth $P$, and multiplier $\chi$ for $\Gamma$.
\end{definition}

\begin{remark}
	As before, one can consider strong higher depth quantum modular forms, looking at asymptotic expansions and not just values. The functions described in this paper satisfy this stronger condition.
\end{remark}

\subsection{Double Eichler Integrals}\label{Section: double Eichler integrals}
In Section \ref{Section: Multiple Eichler Integrals of weight 1} we consider certain double Eichler integrals and investigate their transformation properties. Here, we recall relevant definitions and results.

Let $f_j \in S_{k_j} (\Gamma, \chi_j)$ be a cusp form of weight $k$ with multiplier $\chi_j$ on $\Gamma \subset \text{SL}_2 (\Z)$. If $k_j = \frac{1}{2}$ then we allow $f_j \in M_{\frac{1}{2}} (\Gamma, \chi_j)$, the space of all holomorphic modular forms of weight $\frac{1}{2}$ with multipler $\chi_j$. We define a double Eichler integral by 

\begin{equation*}
I_{f_1, f_2} (\tau) \coloneqq \int_{- \bar{\tau}}^{i \infty} \int_{\omega_1}^{i \infty} \frac{ f_1(\omega_1) f_2(\omega_2) }{ (-i(\omega_1 + \tau))^{2-k_1} (-i(\omega_2 + \tau))^{2-k_2}} d \omega_2 d \omega_1 ,
\end{equation*} 
along with the multiple error of modularity ($\frac{d}{c} \in \Q$)

\begin{equation*}
r_{f_1 , f_2 , \frac{d}{c}} (\tau) \coloneqq \int_{\frac{d}{c}}^{i \infty} \int_{\omega_1}^{\frac{d}{c}} \frac{f_1(\omega_1) f_2(\omega_2) }{ (-i(\omega_1 + \tau))^{2-k_1} (-i(\omega_2 + \tau))^{2-k_2}} d \omega_2 d \omega_1 .
\end{equation*}
In \cite{Higher_depth_QMFs_2} the authors also prove the following proposition, which we will make use of.

\begin{proposition}\label{Proposition: transformation for vector-valued}
	Consider functions $f_j , g_\ell$ $(1 \leq j \leq N , 1 \leq \ell \leq M)$ that are vector-valued modular forms which satisfy the transformations
	
	\begin{equation*}
	\begin{split}
	& f_j \left(- \frac{1}{\tau} \right) = (- i \tau)^{\kappa_1} \sum_{1 \leq k \leq N} \chi_{j,k}^{-1} f_k (\tau) \hspace{10pt} , \hspace{10pt}	g_{\ell} \left(- \frac{1}{\tau}\right) = (- i \tau)^{\kappa_2} \sum_{1 \leq m \leq N} \psi_{\ell,m}^{-1} g_m (\tau) ,
	\end{split}
	\end{equation*}
	where $\kappa_1, \kappa_2 \in \frac{1}{2} + \N_0$. Then we have the transformation formula 
	
	\begin{equation*}
	\begin{split}
	& I_{f_j , g_{\ell}} (\tau) - (-i \tau)^{\kappa_1 + \kappa_2 - 4} \sum_{\substack{1 \leq k \leq N \\ 1 \leq m \leq M}} \chi_{j,k}^{-1} \psi_{\ell,m}^{-1} I_{f_k, g_m}  \left( - \frac{1}{\tau} \right) \\
	= & \int_0^{i\infty} \int_{\omega_1}^{i \infty} \frac{f_j (\omega_1) g_\ell (\omega_2)}{(-i(\omega_1 + \tau))^{2-\kappa_1} (-i(\omega_2 + \tau))^{2-\kappa_2}} d\omega_1 d\omega_2 + I_{f_j} (\tau) r_{g_\ell} (\tau) -  r_{f_j} (\tau) r_{g_\ell} (\tau).
	\end{split} 
	\end{equation*}
	
\end{proposition}

The one-dimensional version of this proposition can be concluded in a similar way, regarding $g_\ell (\tau)$ as constant. In particular, we define

\begin{equation*}
I_{f_j} (\tau) \coloneqq \int_{- \bar{\tau}}^{i \infty} \frac{f_j (\omega)}{(-i ( \omega + \tau))^{2-k}} d \omega, \hspace{20pt} r_{f_j, \frac{d}{d}} (\tau) \coloneqq \int_{\frac{d}{c}}^{\infty} \frac{f_j (\omega)}{(-i ( \omega + \tau))^{2-k}} d \omega.
\end{equation*}
 If $k = \frac{1}{2}$ then we allow $f_j$ to be in $ M_k (\Gamma, \chi)$. The one dimensional Eichler integral $I_{f_j}$ is defined on $\mathbb{H} \cup \mathbb{Q}$, whereas the error of modularity $r_{f_j, \frac{d}{c}} $ exists on all $\R \backslash \{ - \frac{d}{c} \}$ and is real-analytic there. If $f_j$ is a cusp form, then $r_{f_j, \frac{d}{c}} $ exists on all $\R$. The transformation property then follows from the above.
We note that Proposition \ref{Proposition: transformation for vector-valued} implies that the double Eichler integrals above are vector-valued quantum modular forms of depth two.

\subsection{Gauss Sums}\label{Section:Gauss sums}
Here we recall, without proof, some relevant results on the vanishing of quadratic Gauss sums, which we will use when investigating the radial asymptotic behaviour of our function in Section \ref{Section: asymptotic behaviour of F at Certain Roots of Unity}.

Let $a,b,c \in \N$ and denote the generalised quadratic Gauss sum by

\begin{equation*}
G(a,b,c) \coloneqq \sum_{n \pmod{c}} e^{\frac{2 \pi i (an^2 + bn)}{c} }.
\end{equation*}

Then we have the following Lemma, which follows from basic properties of Gauss sums - see e.g. Chapter 1 of \cite{bruce1998gauss}.

\begin{lemma}\label{Lemma: Gauss sums vanishing}
	The following results on the vanishing of $G(a,b,c)$ hold:
	\begin{enumerate}
		\item If $\gcd(a,c) > 1$ and $\gcd(a,c) \nmid b$ then $G(a,b,c) = 0$. 
		\item If $4 \mid c$, $b$ is odd, and $\gcd(a,c) = 1$ then $G(a,b,c) = 0$.
		\item If $c \equiv 2 \pmod{4}$ and $\gcd(a,c) = 1$ then $G(a,0,c) = 0$.
	\end{enumerate}
\end{lemma}

\subsection{Boosted Error Functions}\label{Section: boosted error functions}
In Section \ref{Section: Completed Indefinite theta functions 2} we relate a double Eichler integral to a signature $(2,2)$ indefinite theta function. To do so, we use techniques described in \cite{Generalised_error_functions}. There, the authors consider so-called boosted error functions and use them to find ``modular completions" of a certain family of indefinite theta functions in signature $(n-2,2)$. A modular completion of a non-modular holomorphic function $f(\tau)$ is any function $g(\tau)$ such that $\widetilde{f} (\tau) \coloneqq f(\tau) + g(\tau)$ is modular non-trivially, i.e. $g(\tau)$ is non-holomorphic.

We recall the relevant simplified results here for convenience in signature $(2,2)$, noting in particular the change in notation ``flips" the conditions of the double null limit situation described in Section 4.3 of \cite{Generalised_error_functions}.

Consider a bilinear form $B(x,y) \coloneqq x^T A y$ for a symmetric $m \times m$ matrix $A$, and its associated quadratic form $Q(x) \coloneqq \frac{1}{2} B(x,x)$. Assume that $Q(x)$ has signature $(2,2)$ and also that, for $\mu \in L \subset \Z^4$, we have $Q(\mu) \in \Z$. Take four vectors $C_1, C_2, C_1' , C_2' \in \R^4$. Then we define the orthogonal projections

\begin{equation*}
\begin{split}
C_{1 \perp 2}  \coloneqq C_1 - \frac{B(C_1, C_2)}{Q(C_2)} C_2 \hspace{10pt} \text{ and } \hspace{10pt} C_{2 \perp 1}  \coloneqq C_2 - \frac{B(C_1, C_2)}{Q(C_1)} C_1 ,
\end{split}
\end{equation*}
along with the discriminant $\Delta (C_1 , C_2) \coloneqq Q(C_1)Q(C_2) - B(C_1, C_2)^2 $.

We let $C_{m'} = C_m '$ and let $\Delta_{\mathcal{I}}$ denote the determinant of the Gram matrix $B(C_n, C_m)_{n,m \in \mathcal{I}}$ where $\mathcal{I}$ is a subset of indices $\{1,1',2,2'\}$. Further, let $D_{m,n}$ be off-diagonal cofactors of the Gram matrix $B(C_m, C_n)_{m,n \in \{1,1',2,2'\}}$.

We require

\begin{enumerate}
	\item $B(C_1 , C_2') = B(C_1' , C_2') = B(C_1' , C_2) = 0$,
	\item $Q (C_1) < 0$ and $Q(C_2) < 0$,
	\item $Q(C_1') = Q(C_2') = 0$,
	\item $B(C_1, C_1') < 0$ and $B(C_2, C_2') < 0$,
	\item $\Delta (C_1, C_2) > 0$,
	\item $M_{00}$ is positive definite,
\end{enumerate}
where $M_{00} \coloneqq \left( \begin{smallmatrix} \Delta_{122'} & D_{1'2'} \\  D_{1'2'} & \Delta_{11'2} \end{smallmatrix} \right)$. Then we define boosted complementary error functions in one and two dimensions by

\begin{equation*}
\begin{split}
M(C; x) & \coloneqq M\left(\frac{B (C, x)}{\sqrt{- Q(C)}} \right) ,\\
M_2(C_1, C_2 ; x) & \coloneqq M_2 \left(\frac{- B(C_1 , C_2)}{\sqrt{\Delta(C_1,C_2)}} ; \frac{B(C_2, x)}{\sqrt{- Q(C_2)}} , \frac{B(C_{1 \perp 2} , x)}{\sqrt{ - Q(C_{1 \perp 2})}} \right) .
\end{split}
\end{equation*}
The authors of \cite{Generalised_error_functions} then provide the following Theorem, describing the completion of a certain theta function.

\begin{theorem}\label{Theorem: completion to non-holomorphic}
	Under the conditions above, consider the locally constant function given by 
	\begin{equation*}
	\Phi (x) \coloneqq (\sgn \left(B(C_1,x) \right) - \left( \sgn B(C_1',x) \right) ) (\left( \sgn B(C_2,x) \right) - \left( \sgn B(C_2',x) \right) ) .
	\end{equation*}
	Let $\tau = u + iv \in \mathbb{H}$. Then the theta function 
	
	\begin{equation*}
	\vartheta \left[ \Phi (x) \right] (\tau) \coloneqq \sum_{\lambda \in a + \Z^4} \Phi \left(\sqrt{2 v} \lambda \right) q^{Q(\lambda)} ,
	\end{equation*} 
	admits a modular completion to a non-holomorphic theta series 
	
	\begin{equation*}
	\vartheta \left[ \widehat{\Phi} (x) \right] (\tau) \coloneqq \sum_{\lambda \in a + \Z^4} \widehat{\Phi} \left(\sqrt{2 v} \lambda\right) q^{Q(\lambda)} ,
	\end{equation*} 
	of weight two, where $a \in \R^4$. The completion (in terms of the function $\Phi (x)$) is given by
	
	\begin{equation*}
	\begin{split}
	\hat{\Phi} (x) - \Phi (x) = & M_2 \left(C_1, C_2 ; x\right) + \left(\sgn \left( B(C_{2 \perp 1}, x) \right) - \left( \sgn B(C_2' , x) \right) \right) M\left(C_1 ; x\right) \\
	& + \left(\sgn \left(B(C_{1 \perp 2}, x) \right)- \sgn \left( B(C_1' , x) \right) \right) M\left(C_2 ; x\right) .
	\end{split}
	\end{equation*}
\end{theorem}

\section{The Partial Theta Function $F$}\label{Section: F, J and Q(x)}

Throughout, we consider a positive definite integral binary quadratic form $Q(n) \coloneqq a_1 n_1^2 + a_2 n_1 n_2 + a_3 n_2^2$, where $a_j \in \N$ for $1 \leq j \leq 3$, and $\gcd(a_1, a_2, a_3) = 1$. Let the discriminant of $Q(n)$ be $- D \coloneqq a_2^2 - 4 a_1 a_3 < 0$.

Let $s \geq 1$ be some fixed integer. We write elements of $ \Z[\frac{1}{s}]$ in the form $r + \frac{x}{s}$ where $r,x \in \Z$ and $- \frac{s}{2} \leq x <  \frac{s}{2}$, and let $\alpha^{(k)} \coloneqq (\alpha^{(k)}_{1} , \alpha^{(k)}_{2})$ be a pair of elements in $\Z [ \frac{1}{s}] \times \Z [ \frac{1}{s}] $, labeled by $(k)$. For $s \neq 1$ we require that $\alpha$ does not lie in $\Z^2$ (we could instead add a condition here similar to \eqref{sum of epsilon is 0} for just elements in $\Z^2$. However, this would be equivalent to breaking the vector-valued form into two separate forms, considering those elements in $\Z^2$ in a separate vector, with $s = 1$).

We let

\begin{equation*}\label{definition of J*}
\J^* \coloneqq \left\{ \alpha^{(j)} \mid 1 \leq j \leq N \right\},
\end{equation*}
be a finite set of $N$ such pairs, and define

\begin{equation*}
\J \coloneqq \J^* \cup \left\{ (1 - \alpha^{(j)}) \mid \alpha^{(j)} \in \J^*, 1 \leq j \leq N \right\} ,
\end{equation*}  
where $(1 - \alpha^{(k)}) \coloneqq (1 - \alpha^{(k)}_1 , 1 - \alpha^{(k)}_2)$ is meant componentwise. For convenience, we will often suppress the superscript on elements $\alpha \in \J$. We are free to assume $s$ is minimal, such that each element $\alpha$ has at least one of $x_1 , x_2$ coprime to $s$ (otherwise, we can reduce each fractional part until we are in this situation, possibly splitting into different sets with two different $s_1$ and $s_2$).

We also work with subsets of $\J^*$ given by

\begin{equation*}
\begin{split}
\J^{*}_1  \coloneqq \left\{ \alpha  \in \J^* \mid \alpha_1 \in \Z  \right\} \hspace{10pt}\text{ and } \hspace{10pt} \J^{*}_2  \coloneqq \left\{ \alpha  \in \J^* \mid \alpha_2 \in \Z  \right\}.
\end{split}
\end{equation*}
Consider also a function $\varepsilon \colon \J^* \rightarrow \R \backslash \{ 0 \}$ extended to $\J$ by the relation $\varepsilon (1 - \alpha ) = \varepsilon (\alpha)$, such that

\begin{equation}\label{sum of epsilon is 0}
\sum_{\alpha \in \J} \varepsilon(\alpha) = 0.
\end{equation}

For fixed $\J^*$, $\varepsilon$, and $Q(n) = a_1 n_1^2 + a_2 n_1n_2 + a_3 n_2^2$, the function that we concentrate on in this paper is given by

\begin{equation*}
\begin{split}
F (q) \coloneqq  \sum_{\alpha \in \J} \varepsilon(\alpha) \sum_{ n \in \alpha + \N_0^2 } q^{Q(n)} & -  \frac{1}{2} \sum_{\alpha \in \J^*_1 } \varepsilon(\alpha) \sgn^* (\alpha_1)  \left( \sum_{j \in 1- \alpha_2 + \mathbb{N}_0 }  q^{a_3 j^2} -    \sum_{j \in  \alpha_2 + \mathbb{N}_0 }  q^{a_3 j^2} \right) \\
& - \frac{1}{2} \sum_{\alpha \in \J^*_2}  \varepsilon(\alpha) \sgn^* (\alpha_2) \left( \sum_{j \in 1- \alpha_1 + \mathbb{N}_0 }  q^{a_1 j^2} -   \sum_{j \in  \alpha_1 + \mathbb{N}_0 }  q^{a_1 j^2} \right) .
\end{split}
\end{equation*}

\textit{Throughout, if $\alpha_1 , \alpha_2 \in \Z$ then we omit possible $n = (0,0)$ and $j = 0$ terms in summations implicitly. In each case, this is equivalent to subtracting a constant term and so does not affect modularity properties.}

We consider three different parts separately, writing $F(q) = F_1 (q) + F_2 (q) + F_3(q)$, where

\begin{equation*}\label{definition of F_1}
F_1 (q) \coloneqq \sum_{\alpha \in \J}   \varepsilon(\alpha) \sum_{n \in \alpha + \N_0^2 } q^{Q(n)} .
\end{equation*}

We also define  

\begin{equation*}
\begin{split}
& F_2 (q) \coloneqq  -  \frac{1}{2} \sum_{\alpha \in \J^*_1 } \varepsilon(\alpha) \sgn^* (\alpha_1)  \left( \sum_{j \in 1- \alpha_2 + \mathbb{N}_0 }  q^{a_3 j^2} -    \sum_{j \in  \alpha_2 + \mathbb{N}_0 }  q^{a_3 j^2} \right) , \\
& F_3 (q) \coloneqq  - \frac{1}{2} \sum_{\alpha \in \J^*_2}  \varepsilon(\alpha) \sgn^* (\alpha_2) \left( \sum_{j \in 1- \alpha_1 + \mathbb{N}_0 }  q^{a_1 j^2} -   \sum_{j \in  \alpha_1 + \mathbb{N}_0 }  q^{a_1 j^2} \right) .
\end{split}
\end{equation*}

\begin{remark}
	If for each $(a , x) \in \J^*_1$ the element $ (b , 1-x)$ is also in $\J^*_1$ and $\sgn^* (a) = \sgn^* (b)$ as well as $\varepsilon(a,x) = \varepsilon (b,1-x)$, then the function $F_2$ vanishes identically. A similar statement holds for the function $F_3$.
\end{remark}

Although this definition is only the analogue of the function called $F_1$ from \cite{Higher_depth_QMFs}, it is worth noting that results and techniques therein combined with those of the present paper allow us to also consider the obvious generalization of the function called $F_2$ defined by Bringmann, Kaszian, and Milas in \cite{Higher_depth_QMFs}, and to give analogous results. Again note that we have conflicting notation, and our function $F_2$ is different to that in \cite{Higher_depth_QMFs}.\\

\begin{remark}
	It is possible to drop the condition \eqref{sum of epsilon is 0} if we are willing to lose the possibility of having quantum set $\mathbb{Q}$. This is essentially the same as using a vector of quantum sets - one for each fixed $\alpha \in \J^*$ - such that the main term in the Euler-Maclaurin summation formula of the element $\sum_{ n \in \alpha + \N_0^2 } q^{Q(n)} + \sum_{ n \in 1 - \alpha + \N_0^2 } q^{Q(n)}$ vanishes at certain roots of unity dictated by the quantum set. In this case, the largest $\text{SL}_2 (\Z)$-invariant quantum set would be $\mathcal{Q}_1$ for one fixed $\alpha$, as defined in the following section. However, this would be empty in some cases, and we would need to work on suitable subgroups of the full modular group to restore an infinite quantum set.
\end{remark}

\begin{example}
	As a running example we consider the positive definite quadratic form $Q(x) = 2x_1^2 + x_1x_2 + x_2^2$ of discriminant $-D = -7$, along with the set
	\begin{equation*}
	\J^* = \left\{ \left( \frac{1}{4} , \frac{1}{4} \right) , \left( \frac{1}{4}, -\frac{2}{4}  \right) \right\}.
	\end{equation*}
	Further, we set 
	
	\begin{equation*}
	\varepsilon(\alpha) = \begin{dcases}
	1 &\text{ if } \alpha = \left( \frac{1}{4} , \frac{1}{4} \right), \\
	-1 &\text{ if } \alpha = \left( -\frac{2}{4} , \frac{1}{4} \right).
	\end{dcases}
	\end{equation*}
	We see that this set satisfies our condition with $s = 4$, and that both $\J_1^* , \J_2^*$ are empty, so we need only consider $F_1$.
\end{example}

\section[The Quantum Set]{The Quantum Set $\mathcal{Q}$}\label{Section: quantum set}

Here we describe the quantum set for our function $F$, the main idea being that the choice of set will force the main term in the Euler-Maclaurin summation formula to vanish so that we do not obtain a growing term in the asymptotic expansions toward certain points in Section \ref{Section: asymptotic behaviour of F at Certain Roots of Unity}. 

Throughout we write elements of $\mathbb{Q}$ as $h/k$ with $\gcd(h,k) = 1$, and define $\delta \coloneqq \gcd(h,s)$ and $\gamma \coloneqq \gcd(k,s)$. For a fixed $\alpha = (r_1, r_2) + (x_1, x_2)/s \in \J$, set

\begin{equation*}
g(x) = g(x_1 , x_2) \coloneqq \begin{dcases}
\gcd(2a_1 x_1 + a_2 x_2, a_2 x_1 + 2a_3 x_2) & \text{ if } x \neq (0,0), \\
1 &\text{ if } x = (0,0),
\end{dcases}  
\end{equation*} 
with the convention that $\gcd(0,t) = t$ for $t \in \N_0$. We define

\begin{equation*}
\begin{split}
& G \coloneqq \left\{ g(x) \mid \alpha \in \J \right\}.
\end{split}
\end{equation*}
Then the first part of the quantum set is given by
\begin{equation*}
\mathcal{Q}_1 \coloneqq \left\{ \frac{h}{k} \Bigm| \frac{s}{\delta}, \frac{s}{\gamma}  \nmid g \text{ for every $g \in G$ } \right\}.
\end{equation*}

In particular, note that if $g(x) = 1$ for every choice of $\alpha$ then the conditions on $\frac{s}{\delta}$ and $\frac{s}{\gamma}$ are always satisfied away from $h \in s \Z$ or $k \in s\Z$. We also differentiate cases based upon whether or not the following congruence condition holds
\begin{equation}\label{Equation: conditions for quantum set}
Q(x) \pmod{s} \hspace{10pt} \text{ is constant across $\J$} .
\end{equation} 
If in addition we have $s \nmid g$ for every $g \in G$, we set
\begin{equation}\label{Equation: first extra quantum set}
\begin{dcases*}
\mathcal{Q}_2 \coloneqq \left\{ \frac{h}{k} \Bigm| h \in s \Z  \right\} \text{ and } \mathcal{Q}_3 \coloneqq \left\{ \frac{h}{k} \Bigm| k \in s \Z  \right\} & \text{ if \eqref{Equation: conditions for quantum set} holds}, \\
\mathcal{Q}_2 \coloneqq \left\{ \frac{h}{k} \Bigm| h \in s^2 \Z  \right\} \text{ and } \mathcal{Q}_3 \coloneqq\left\{ \frac{h}{k} \Bigm| k \in s^2 \Z  \right\} & \text{else}.
\end{dcases*}
\end{equation}
If $s$ divides some element in $G$ then the situation is more complicated and we will need to differentiate several cases. In this case, it will be easier to define the ``extra" quantum sets $\mathcal{Q}_2$ and $\mathcal{Q}_3$ algorithmically after \eqref{Equation: sum on l in main term of EM} is introduced and investigated. In each case, we find a particular $n \in \N$ and define
\begin{equation*}
\mathcal{Q}_2 \coloneqq \left\{ \frac{h}{k} \Bigm| h \in s^n \Z \right\} \text{ and } \mathcal{Q}_3 \coloneqq \left\{ \frac{h}{k} \Bigm| k \in s^n \Z \right\}.
\end{equation*}

The ``full" quantum set is then defined as $\mathcal{Q} \coloneqq \mathcal{Q}_1 \cup \mathcal{Q}_2 \cup \mathcal{Q}_3$. Notice that for some choices of $s, Q, \J$ this quantum set is somewhat sparse (e.g. if $s = 2,4$ and $a_2 \in 2\Z$ and $2^k \in G$ for some $k \in \N$). 

The transformation formulae of the double Eichler integrals in Section \ref{Section: Multiple Eichler Integrals of weight 1} currently require $\mathcal{Q} = S \mathcal{Q}$, and so we note the following equalities. First, the action of $S$ on a fraction $\frac{h}{k}$ is given by $S(\frac{h}{k}) = \frac{-k}{h}$. Then it is clear that $S \mathcal{Q}_1 = \mathcal{Q}_1$, since each of the numerator and denominator are assumed to have the same property in the definition of $\mathcal{Q}_1$ above. Then notice that, for a fixed choice of $n \in \N$, we have $S \mathcal{Q}_2 = \mathcal{Q}_3$ as we just switch the numerator and denominator, i.e. for $h \in s^n\Z$ we have that $\gcd(k,s) = 1$ and so $S(\frac{h}{k}) = \frac{-k}{h}$ has a denominator lying in $s^n\Z$ and a numerator co-prime to the denominator by construction. The argument is similar as to why $S \mathcal{Q}_3 = \mathcal{Q}_2$. Hence overall we have that $S\mathcal{Q} = \mathcal{Q}$.

\begin{example}\textit{(continued)}
	Continuing our example, we compute $\gcd(2x_1 + x_2, x_1 + 2x_2)$ for each of the elements $x \in \left\{ (1,1) , (1, -2) , (-1,-1) , (-1,2) \right\}$ and find that $G = \left\{1 \right\}$. Furthermore, we have that $Q(x) \equiv 0 \pmod{4}$ for every element in $\J$, hence we take the quantum set $\Q$.
\end{example}

\section{Radial Asymptotic Behaviour of $F$ at Certain Roots of Unity}\label{Section: asymptotic behaviour of F at Certain Roots of Unity}
In this section we aim to find the asymptotic behaviour of the function $F$ at a point $e^{2 \pi i \frac{h}{k} - t}$ as $t \rightarrow 0^+$, with $h/k \in \mathcal{Q}$. To do so, we rewrite $F$ in a way that we may apply the Euler-Maclaurin summation formula.

\subsection{Asymptotic Behaviour of $F_1$}
Decomposing $F$ as above, we concentrate firstly on $F_1$. We have

\begin{equation*}
F_1 \left(e^{2 \pi i \frac{h}{k} - t} \right) = \sum_{\alpha \in \J} \varepsilon(\alpha) \sum_{n \in \alpha + \N_0^2} e^{\left(2 \pi i \frac{h}{k} - t \right) Q(n) } = \sum_{\alpha \in \J}  \varepsilon(\alpha) \sum_{n \in \N_0^2} e^{\left(2 \pi i \frac{h}{k} - t\right){Q(n + \alpha)}}.
\end{equation*} 
Letting $n \mapsto \ell + n \frac{ks}{\delta}$ where $0 \leq \ell_j \leq \frac{ks}{\delta} - 1$ and $\delta \coloneqq \gcd(h,s)$ gives that the sum on $n$ is 

\begin{equation}\label{full sum of F_1}
\sum_{\substack{0 \leq \ell \leq \frac{ks}{\delta} - 1 \\  n  \in \N_0^2}}  e^{\left(2 \pi i \frac{h}{k} - t\right){Q(\ell + n\frac{ks}{\delta} + \alpha)}}.
\end{equation}
Noting that, for $n \in \Z$, we have $ e^{2 \pi i \frac{h}{k} Q(\ell + n\frac{ks}{\delta} + \alpha)} = e^{ 2 \pi i \frac{h}{k} Q(\ell + \alpha)}$ which  is independent of $n$, we can write \eqref{full sum of F_1} as 

\begin{equation*}
\begin{split}
\sum_{0 \leq \ell \leq \frac{ks}{\delta} - 1} e^{ 2 \pi i \frac{h}{k} Q(\ell + \alpha)} \sum_{n \in \frac{\delta}{ks} (\ell + \alpha) + \N_0^2} e^{- t Q(n \frac{ks}{\delta})}.
\end{split}
\end{equation*}
Defining $\mathcal{F}_1 (x) \coloneqq e^{- Q(x)}$ we can therefore write 
\begin{equation}\label{Equation: rewriting F_1 to apply E-M}
F_1 \left( e^{2 \pi i \frac{h}{k} - t} \right) = \sum_{\alpha \in \J} \varepsilon (\alpha)  \sum_{0 \leq \ell \leq \frac{ks}{\delta} - 1} e^{ 2 \pi i \frac{h}{k} Q(\ell + \alpha)} \sum_{n \in \frac{\delta}{ks} (\ell + \alpha) + \N_0^2} \mathcal{F}_1 \left(\frac{ks}{\delta} \sqrt{t} n \right).
\end{equation}
The main term in the Euler-Maclaurin summation formula \eqref{Euler-Maclaurin summation formula} is then given by

\begin{equation}\label{Main term in Euler Maclaurin of F_1}
\frac{\delta^2}{k^2 s^2 t} \mathcal{I}_{\mathcal{F}_1} \sum_{\alpha \in \J} \varepsilon(\alpha) \sum_{0 \leq \ell \leq \frac{ks}{\delta} -1} e^{2 \pi i \frac{h}{k} Q(\ell+\alpha)} ,
\end{equation}
which we will show below vanishes for  $h/k \in \mathcal{Q}$. We may let $\ell$ run modulo $\frac{ks}{\delta}$, since we have that $e^{2\pi i \frac{h}{k} Q(\ell + n \frac{ks}{\delta} + \alpha )} = e^{2 \pi i \frac{h}{k} Q(\ell + \alpha)}$ whenever $n$ is a pair of integers. Hence the sum on $\ell$ equals (writing $\alpha = r + \frac{x}{s}$ as in Section \ref{Section: F, J and Q(x)})

\begin{equation}\label{Equation: sum on l in main term of EM}
\sum_{ \ell \pmod{\frac{ks}{\delta}}} e^{2 \pi i \frac{h}{k} Q \left( \ell+\frac{x}{s} \right)}.
\end{equation}

If $s = 1$ then \eqref{Equation: sum on l in main term of EM} is clearly independent of $\alpha$ and hence the main term vanishes. If $s > 1$ we let $\ell = N + \nu k$, again meant componentwise, with $N \pmod{k}$ and $\nu \pmod{\frac{s}{\delta}}$. Similar to previous calculations, we compute that the sum on $\ell$ in \eqref{Main term in Euler Maclaurin of F_1} is equal to

\begin{equation*}\label{sum on n mod k and nu mod p/delta}
\begin{split}
& \sum_{N \pmod{k}} e^{ \frac{2 \pi i h}{k s^2}\left( (a_1 ( s^2 N_1^2 + 2 s N_1 x_1 ) + a_2 ( s^2 N_1 N_2 + s x_2 N_1 + s x_1 N_2) + a_3 ( s^2 N_2^2 + 2 s x_2 N_2) + Q(x) \right)} \\
\times & \sum_{\nu \pmod{ \frac{s}{\delta}}} e^{2 \pi i \frac{h}{s}  \left( \nu_1 (2 a_1 x_1 + a_2 x_2) + \nu_2 (a_2 x_1 + 2 a_3 x_2) \right)} .
\end{split} 
\end{equation*}
Given this, we see that showing \eqref{Main term in Euler Maclaurin of F_1} is zero becomes equivalent to showing that the following expression vanishes

\begin{equation}\label{Equation: sum on n and nu}
\begin{split}
& \sum_{\alpha \in \J} \varepsilon (\alpha) \sum_{N \pmod{k}} e^{ \frac{2 \pi i h}{k s^2}\left( (a_1 ( s^2 N_1^2 + 2 s N_1 x_1 ) + a_2 ( s^2 N_1 N_2 + s x_2 N_1 + s x_1 N_2) + a_3 ( s^2 N_2^2 + 2 s x_2 N_2) + Q(x) \right)} \\
& \times \sum_{\nu \pmod{ \frac{s}{\delta}}} e^{2 \pi i \frac{h / \delta}{s / \delta}  \left( \nu_1 (2 a_1 x_1 + a_2 x_2) + \nu_2 (a_2 x_1 + 2 a_3 x_2) \right)} .
\end{split} 
\end{equation}

 First, consider values $h/k \in \mathcal{Q}_1$. Since $\gcd(\frac{h}{\delta} , \frac{s}{\delta}) = 1$, the sum on $\nu$ vanishes unless $\frac{s}{\delta} \mid (2 a_1 x_1 + a_2 x_2) $ and $\frac{s}{\delta} \mid (a_2 x_1 + 2 a_3 x_2)$, implying that $\frac{s}{\delta} \mid g(x)$. By construction, this is a contradiction to our assumption on $\mathcal{Q}_1$. Therefore, for $h/k \in \mathcal{Q}_1$ the main term vanishes. It is also easily seen from here that if $s \nmid g(x)$ for each $\alpha$ then the main term will vanish.

Then it remains to show that the main term is zero for our different choices of sets $\mathcal{Q}_2$ and $\mathcal{Q}_3$. First, we consider values of $h \in s\Z$ or $h \in s^2 \Z$ depending on whether \eqref{Equation: conditions for quantum set} holds or not, respectively. For ease of exposition we show this only in the case where \eqref{Equation: conditions for quantum set} holds - the second case follows similarly.

Writing $Q(x) = sX + x_0$ for each choice of $\alpha$, with $0 \leq x_0 < s$ constant across $\J$ by assumption, it follows that it suffices to show

\begin{equation*}\label{final step in main term is 0}
\begin{split}
\sum_{\alpha \in \J} \varepsilon (\alpha) \sum_{N \pmod{k}}  e^{ \frac{2 \pi i h/s}{k}\left( (a_1 ( s N_1^2 + 2 N_1 x_1 ) + a_2 ( s N_1 N_2 + x_2 N_1 + x_1 N_2) + a_3 ( s N_2^2 + 2 x_2 N_2) + X \right)} = 0.
\end{split} 
\end{equation*}
Since $\delta = s$ in $\mathcal{Q}_2$ we see that $\gcd(k,s) = 1$ and so in particular the inverse of $s$ modulo $k$, which we denote by $\bar{s}$, exists. Making the change of variables $N \mapsto N - \bar{s} x$ and using that $h/s \in \Z$ gives

\begin{equation*}
\begin{split}
e^{ \frac{- 2 \pi i \bar{s} x_0 h/s }{k}} \sum_{N \pmod{k}} e^{ \frac{2 \pi i h}{k} Q(N)} \sum_{\alpha \in \J} \varepsilon (\alpha) ,
\end{split} 
\end{equation*}
which vanishes since $\sum_{\alpha \in \J} \varepsilon (\alpha) = 0$. 

Next, consider elements in $\mathcal{Q}_3$. First, fix a choice of $\alpha$. Note that $s$ does not divide both $x_1 , x_2$ for each $\alpha$ by assumption (if this were not the case, then \eqref{Equation: sum on l in main term of EM} is seen to be constant and so we would require the same constant across all choices of $\alpha$, implying that every $\alpha$ lies in $\Z^2$). We have already seen that when $s \nmid g(x)$ the main term vanishes for any value of $k \in \Z$, and so we now assume $s \mid g(x)$ (a similar argument holds for the cases where $s/\delta \mid g(x)$).

We are aiming to find $m \in \N$ such that for $k \in s^m \Z$ the term

\begin{equation*}
\begin{split}
\sum_{N \pmod{k}}  e^{ \frac{2 \pi i h}{s k}\left( a_1 ( s N_1^2 + 2 N_1 x_1 ) + a_2 ( s N_1 N_2 + x_2 N_1 + x_1 N_2) + a_3 ( s N_2^2 + 2 x_2 N_2) \right)}
\end{split} 
\end{equation*}
vanishes (here we have taken the factor $e^{\frac{2 \pi i h Q(x)}{k s^2}}$ out of the sum). Since $s \mid g(x)$ we may define $X_1 \coloneqq (2 a_1 x_1 + a_2 x_2)/s$ and $X_2 \coloneqq (a_2 x_1 + 2 a_3 x_2)/s$, each of which lie in $\Z$, to obtain the two-dimensional Gauss sum (putting $k = s^m$)

\begin{equation}\label{Equation: Gauss sum k = s^n}
\begin{split}
\sum_{N \pmod{s^m}}  e^{ \frac{2 \pi i h}{s^m}\left( a_1 N_1^2 + a_2  N_1 N_2 + a_3  N_2^2 + N_1 X_1 + N_2 X_2  \right)} .
\end{split} 
\end{equation}

The main idea here is to reduce this to a product of one-dimensional Gauss sums and use well-known results. As with most Gauss sums we may reduce \eqref{Equation: Gauss sum k = s^n} to the product of two-dimensional Gauss sums over prime powers (via the Chinese Remainder theorem), and hence consider 

\begin{equation}\label{Equation: Gauss sum k = p^n}
\begin{split}
\sum_{N \pmod{p^n}}  e^{ \frac{2 \pi i h}{p^n}\left( a_1 N_1^2 + a_2  N_1 N_2 + a_3  N_2^2 + N_1 X_1 + N_2 X_2 \right)} ,
\end{split} 
\end{equation}
where $p$ is some prime dividing $s$, $n \in \N$, and $\gcd(h,p) = 1$ by construction. We see that for the main term in the Euler-Maclaurin expansion formula to vanish, it suffices to show that the above sum is zero for any prime dividing $s$. Since $s$ does not divide both $x_1$ and $x_2$ there exists at least one $p^{\ell} \parallel s$ that does not divide both $x_1$ and $x_2$. Fixing such a prime, we see that at least one of $a_1$ and $a_3$ admit an inverse modulo $p^n$. For ease of exposition we assume throughout that $\bar{a}_1$ exists, denoting the inverse of $a_1$. Next, we differentiate situations depending on the parity of $p$.

If $p$ is odd then $\bar{2}$ also exists modulo $p^n$. Completing the square on $N_1$ in the exponential term of \eqref{Equation: Gauss sum k = p^n} gives

\begin{equation*}
\begin{split}
& a_1 N_1^2 + a_2 N_1 N_2 + a_3 N_2^2 + N_1 X_1 + N_2 X_2  \\ 
& \equiv  a_1 (N_1 + \bar{2} \bar{a}_1 (a_2 N_2 + X_1))^2 
 - \bar{4} \bar{a}_1 (a_2 N_2 + X_1)^2 + a_3 N_2^2 + X_2 N_2 \pmod{p^n}.
\end{split}
\end{equation*}
Thus the sum on $N$ becomes (up to constants, after making the shift $N_1 \mapsto N_1 + \bar{2} \bar{a}_1 (a_2 N_2 + X_1) $)

\begin{equation*}
\sum_{N_1 \pmod{p^n}} e^{ \frac{2 \pi i h }{p^n} a_1 N_1^2 }  \sum_{N_2 \pmod{p^{n}}} e^{ \frac{2 \pi i h }{p^{n}} \left( D_1 N_2^2 + 2  x_2^* N_2   \right) } ,
\end{equation*}
where $D_1 \coloneqq  a_3 - \bar{4} \bar{a}_1 a_2^2$ and $x_2^* \coloneqq \frac{x_2 (a_3 - \bar{4} \bar{a}_1 a_2^2)}{p^\ell \prod_j q_j^{n_j}} \in \Z $, with $s = p^\ell \prod_j q_j^{n_j}$ written in its prime decomposition. This is now a product of one-dimensional quadratic Gauss sums. Concentrating on the sum on $N_2$ we consider first the case where $D_1$  is not coprime with $p^n$. That is, we assume $\gcd(D_1, p^n) = p^r$ with $r \geq 1$. Then, if $p^r \nmid x_2^*$ we see by part one of Lemma \ref{Lemma: Gauss sums vanishing} that the Gauss sum vanishes. The second case is where $D_1$ is coprime with $p^n$, implying that $p^\ell \mid x_2$ since $x_2^* \in \Z$. Thus the sum vanishes for any $n \in \N$ unless $p^{\ell} \mid x_2$, which cannot happen as then we would have that $p^{\ell}$ divides both $x_1$ and $x_2$ since $p \mid g(x)$. Hence the sum vanishes for all $k \in s \Z$.

Next we turn to the case of $p = 2$. In particular, we then have that $2 \mid a_2$. Again we have that at least one of $\bar{a}_1$ or $\bar{a}_3$ exist modulo $2^n$ for $n \in \N$, and we assume that $\bar{a}_1$ does for ease of exposition. Letting $N_1 \rightarrow N_1 - \bar{a}_1 \frac{a_2}{2} N_2$ gives us the numerator

\begin{equation*}
a_1 N_1^2 + X_1 N_1 + N_2^2 \left(a_3 - \bar{a}_1 \frac{a_2^2}{4} \right) + N_2 \left(X_2 - X_1 \bar{a}_1 \frac{a_2}{2} \right) \pmod{2^n} ,
\end{equation*}
and so again we obtain a product of two one-dimensional Gauss sums. Explicitly, we have

\begin{equation*}
\sum_{N_1 \pmod{2^n}} e^{ \frac{2 \pi i h}{2^n} \left(a_1 N_1^2 + X_1 N_1\right) }  \sum_{N_2 \pmod{2^n}} e^{ \frac{2 \pi i h}{2^n} \left(D_2 N_2^2  + N_2 \left(X_2 - X_1 \bar{a}_1 \frac{a_2}{2} \right) \right) },
\end{equation*}
with $D_2 \coloneqq a_3 - \bar{a}_1 \frac{a_2^2}{4}$. If $X_1$ is odd then the sum on $N_1$ will vanish for any $n \geq 2$ using part two of Lemma \ref{Lemma: Gauss sums vanishing}, so take $k \in s \Z$ with $4 \mid k$. If $X_1$ is even then we put $N_1 \rightarrow N_1 + X_1/2$ to give the sum on $N_1$ as (up to a constant)

\begin{equation*}
\sum_{N_1 \pmod{2^n}} e^{ \frac{2 \pi i h}{2^n} a_1 N_1^2 } .
\end{equation*}
Using part three of Lemma \ref{Lemma: Gauss sums vanishing} this vanishes only if $n = 1$, and so we may choose $k \in s\Z$ with $2 \parallel k$. The sum on $N_2$ is

\begin{equation*}
\sum_{N_2 \pmod{2^n}} e^{ \frac{2 \pi i h}{2^{n}} \left( D_2 N_2^2 + 2 x_2^* N_2 \right) } ,
\end{equation*}
where $x_2^* \coloneqq \frac{D_2 x_2}{ 2^{\ell} \prod_j q_j^{n_j}}$ and $2^{\ell} \parallel s$. In a similar fashion to the case of odd $p$, this vanishes unless $2^{\ell - 1} \mid x_2$. In this case, let $r$ be such that $2^r \parallel D_2$ and put $t \coloneqq D_2 / 2^r$ odd, so the sum becomes 

\begin{equation*}
2^r \sum_{N_2 \pmod{2^{n-r}}} e^{ \frac{2 \pi i h}{2^{n-r}} \left(t N_2^2 + b N_2 \right) } .
\end{equation*}
If $b \coloneqq \frac{2 t x_2}{2^{\ell} \prod_j q_j^{n_j} }$ is odd then this vanishes for all $ n \geq r + 2$ by part two of Lemma \ref{Lemma: Gauss sums vanishing}, so choose $k \in s \Z$ with $2^{r+2} \mid k$. If it is even then we complete the square by shifting $N_2 \rightarrow N_2 - \bar{t} b/2$ to obtain (up to a constant)

\begin{equation*}
\sum_{N_2 \pmod{2^{n-r}}} e^{ \frac{2 \pi i h t N_2^2 }{2^{n-r}} } ,
\end{equation*}
vanishing only when $n = r + 1$. Here we may then choose $k \in s\Z$ with $2^{r+1} \parallel k$ by part three of Lemma \ref{Lemma: Gauss sums vanishing}.

Choosing overall the minimum $n$ such that the main term vanishes for each choice of $\alpha \in \J$ (and ensuring that it is at least $1$ or $2$ depending on \eqref{Equation: conditions for quantum set}), we may form the extra part of the quantum set. Analogously to \eqref{Equation: first extra quantum set} we define

\begin{equation*}
\mathcal{Q}_2 \coloneqq \left\{ \frac{h}{k} \Bigm| h \in s^n \Z  \right\} \text{ and } \mathcal{Q}_3 \coloneqq \left\{ \frac{h}{k} \Bigm| k \in s^n \Z  \right\} ,
\end{equation*}
along with any necessary conditions on whether higher powers of $s$ may divide $h$ or $k$.

Continuing with analysing the asymptotic behaviour of $F_1$ we next turn to the other terms in the Euler-Maclaurin summation formula \eqref{Euler-Maclaurin summation formula}. With $h/k \in \mathcal{Q}$, the second term is given by

\begin{equation*}
\sum_{\alpha \in \J} \varepsilon(\alpha) \sum_{0 \leq \ell \leq \frac{ks}{\delta} -1} e^{2 \pi i \frac{h}{k} Q(\ell+\alpha)} \sum_{n_2 \geq 0} \frac{B_{n_2 + 1} \left( \frac{\delta ( \ell_2 + \alpha_2)}{ks} \right)}{(n_2 + 1)!} \int_0^\infty \mathcal{F}_1^{ (0, n_2)} (x_1, 0) dx_1 \left( \frac{ks \sqrt{t}}{\delta} \right)^{n_2 - 1} .
\end{equation*}

In the same way as in \cite{Higher_depth_QMFs} we claim that the terms where $n_2$ is even vanish. To see this we first recall that each $\alpha \in \J$ pairs canonically with $1- \alpha$ by construction. Thus if we show that the expression 

\begin{equation*}
\sum_{0 \leq \ell \leq \frac{ks}{\delta} -1} \left( e^{2 \pi i \frac{h}{k} Q(\ell+\alpha)} B_{2 n_2 + 1} \left( \frac{\delta ( \ell_2 + \alpha_2)}{ks} \right) + e^{2 \pi i \frac{h}{k} Q(\ell+1-\alpha)} B_{2 n_2 + 1} \left(\frac{\delta ( \ell_2 + 1- \alpha_2)}{ks}\right) \right) 
\end{equation*}
vanishes, then the claim will follow immediately. Recalling the behaviour of the Bernoulli polynomials \eqref{equation: Bernoulli polynomial transformation} and shifting the second term via $\ell \mapsto -\ell + (-1 + \frac{ks}{\delta}) (1,1)$ gives this immediately.

Treating the terms where $n_2$ is odd (again using the canonical pairing in $\J$) we now see that the Bernoulli polynomial transform no longer cancels, but give the same contribution. Hence the second term in the Euler-Maclaurin summation formula for $F_1$ is

\begin{equation}\label{Equation: second term of asymptotics}
\begin{split}
-2 \sum_{\alpha \in \J^*} \varepsilon (\alpha) \sum_{0 \leq \ell \leq \frac{ks}{\delta} -1} e^{2 \pi i \frac{h}{k} Q(\ell+\alpha)} & \sum_{n_2 \geq 0} \frac{B_{2n_2 + 2} \left( \frac{\delta ( \ell_2 + \alpha_2)}{ks} \right)}{(2 n_2 + 2)!}\int_0^\infty \mathcal{F}_1^{(0, 2n_2 +1)} (x_1 ,0) dx_1 \left( \frac{k^2 s^2 t}{\delta^2} \right)^{n_2}.
\end{split}
\end{equation}
Similarly, the third term in \eqref{Euler-Maclaurin summation formula} is given by 

\begin{equation*}
\begin{split}
-2 \sum_{\alpha \in \J^*} \varepsilon (\alpha) \sum_{0 \leq \ell \leq \frac{ks}{\delta} -1} e^{2 \pi i \frac{h}{k} Q(\ell+\alpha)} & \sum_{n_1 \geq 0} \frac{B_{2n_1 + 2} \left( \frac{\delta ( \ell_1 + \alpha_1)}{ks} \right)}{(2 n_1 + 2)!} \int_0^\infty \mathcal{F}_1^{(2n_1 +1,0)} (0, x_2) dx_2 \left( \frac{k^2 s^2 t}{\delta^2} \right)^{n_1} .
\end{split}
\end{equation*}
The final term of \eqref{Euler-Maclaurin summation formula} is equal to 

\begin{equation*}
\begin{split}
\sum_{\alpha \in \J} \varepsilon (\alpha) & \sum_{0 \leq \ell \leq \frac{ks}{\delta} -1} e^{2 \pi i \frac{h}{k} Q(\ell+\alpha)} \\
\times & \sum_{n_1 , n_2 \geq 0} \frac{B_{n_1 + 1} \left( \frac{\delta ( \ell_1 + \alpha_1)}{ks} \right)}{(n_1 + 1)!} \frac{B_{n_2 + 1} \left( \frac{\delta ( \ell_2 + \alpha_2)}{ks} \right)}{(n_2 + 1)!} \mathcal{F}_1^{(n_1, n_2)} (0,0) \left( \frac{k s \sqrt{t}}{\delta} \right)^{n_1 + n_2} .
\end{split}
\end{equation*}
Proceeding in the same way, only the terms where $n_1 \equiv n_2 \pmod{2}$ are non-zero. Therefore this is equal to 

\begin{equation*}
\begin{split}
2 \sum_{\alpha \in \J^*} \varepsilon (\alpha) & \sum_{0 \leq \ell \leq \frac{ks}{\delta} -1} e^{2 \pi i \frac{h}{k} Q(\ell+\alpha)} \\
\times & \sum_{\substack{n_1 , n_2 \geq 0 \\ n_1 \equiv n_2 \pmod{2}}} \frac{B_{n_1 + 1} \left( \frac{\delta ( \ell_1 + \alpha_1)}{ks} \right)}{(n_1 + 1)!} \frac{B_{n_2 + 1} \left( \frac{\delta ( \ell_2 + \alpha_2)}{ks} \right)}{(n_2 + 1)!} \mathcal{F}_1^{(n_1, n_2)} (0,0) \left( \frac{k s \sqrt{t}}{\delta} \right)^{n_1 + n_2} .
\end{split}
\end{equation*}

\subsection{Asymptotic Behaviour of $F_2$ and $F_3$}

We now focus on the function $F_2$, and use similar techniques to above. Set $\mathcal{F}_2 (x) \coloneqq e^{- a_3 x^2}$, rewrite as in \eqref{Equation: rewriting F_1 to apply E-M}, and use the Euler-Macluarin summation formula in one dimension \eqref{Euler-Maclaurin summation formula one dim} to obtain the main term as

\begin{equation*}
- \frac{\delta}{2 k s \sqrt{t}} \mathcal{I}_{\mathcal{F}_2}  \sum_{\alpha \in \J^*_1} \varepsilon (\alpha) \sgn^* (\alpha_1) \left( \sum_{0 \leq r \leq \frac{ks}{\delta} - 1} e^{2 \pi i \frac{h}{k} a_3 (r + 1 - \alpha_2)^2 } - e^{2 \pi i \frac{h}{k} a_3 (r + \alpha_2)^2} \right) .
\end{equation*}

Letting $r \mapsto \frac{ks}{\delta} - r - 1$ in the first term of the inner summand shows that this vanishes identically. The second term in the one-dimensional Euler-Maclaurin formula for $F_2$ is given by (pairing even terms and noting odd terms vanish as above)

\begin{equation*}
\begin{split}
\frac{1}{2} & \sum_{\alpha \in \J^*_1} \varepsilon(\alpha) \sgn^* (\alpha_1) \sum_{0 \leq r \leq \frac{ks}{\delta} - 1} e^{\frac{2 \pi i h}{k} a_3 (r + (1 - \alpha_2))^2} \sum_{m \geq 0} \frac{ B_{2m+1} \left( \frac{\delta (r + (1 - \alpha_2))}{ks} \right)}{(2m+1)!} \mathcal{F}_2^{(2m)} (0) \left( \frac{k^2 s^2 t}{\delta^2} \right)^{m} \\
& - \frac{1}{2} \sum_{\alpha \in \J^*_1} \varepsilon(\alpha) \sgn^* (\alpha_1) \sum_{0\leq r \leq \frac{ks}{\delta} - 1} e^{\frac{2 \pi i h}{k} a_3 (r + \alpha_2)^2} \sum_{m \geq 0} \frac{ B_{2m+1} \left( \frac{\delta (r + \alpha_2)}{ks} \right)}{(2m+1)!} \mathcal{F}_2^{(2m)} (0) \left( \frac{k^2 s^2 t}{\delta^2} \right)^{m} .
\end{split}
\end{equation*}
The same argument runs for the function $F_3$ with setting $\mathcal{F}_3 (x) \coloneqq e^{- a_1 x^2}$, yielding

\begin{equation*}
\begin{split}
\frac{1}{2} & \sum_{\alpha \in \J^*_2} \varepsilon(\alpha) \sgn^* (\alpha_2) \sum_{0 \leq r \leq \frac{ks}{\delta} - 1} e^{\frac{2 \pi i h}{k} a_1 (r + (1 - \alpha_1))^2} \sum_{m \geq 0} \frac{ B_{2m+1} \left( \frac{\delta (r + (1 - \alpha_1))}{ks} \right)}{(2m+1)!} \mathcal{F}_3^{(2m)} (0) \left( \frac{k^2 s^2 t}{\delta^2} \right)^{m} \\
& - \frac{1}{2} \sum_{\alpha \in \J^*_2} \varepsilon(\alpha) \sgn^* (\alpha_2) \sum_{0\leq r \leq \frac{ks}{\delta} - 1} e^{\frac{2 \pi i h}{k} a_1 (r + \alpha_1)^2} \sum_{m \geq 0} \frac{ B_{2m+1} \left( \frac{\delta (r + \alpha_1)}{ks} \right)}{(2m+1)!} \mathcal{F}_3^{(2m)} (0) \left( \frac{k^2 s^2 t}{\delta^2} \right)^{m} .
\end{split}
\end{equation*}

\section{Double Eichler Integrals of Weight One}\label{Section: Multiple Eichler Integrals of weight 1}
Here we introduce and study a family of double Eichler integrals of weight $1$, and show that they are a part of a vector-valued quantum modular form of depth two and weight one.

Recalling that $Q(n)$ has non-zero coefficients $a_j$ and has discriminant $-D < 0$, for $\alpha \in \J^*$, $\omega_j \in \H$ we set

\begin{equation}\label{Equation: definition of E (tau)}
\mathcal{E}_{\alpha} (\tau) \coloneqq - \frac{\sqrt{D}}{4} \int_{- \bar{\tau}}^{i \infty} \int_{\omega_1}^{i \infty} \frac{\theta_1 (\alpha; \omega_1, \omega_2) + \theta_2 (\alpha; \omega_1, \omega_2)}{\sqrt{-i (\omega_1 + \tau)} \sqrt{-i(\omega_2 + \tau)}} d\omega_2 d\omega_1 ,
\end{equation}
along with theta functions

\begin{equation*}
\theta_1 (\alpha; \omega_1, \omega_2) \coloneqq \frac{1}{a_1} \sum_{n \in \alpha + \Z^2} (2 a_1 n_1 + a_2 n_2) n_2 e^{\frac{\pi i (2 a_1 n_1 + a_2 n_2)^2 \omega_1}{2 a_1} + \frac{ \pi i D  n_2^2 \omega_2}{2 a_1}}
\end{equation*}
and

\begin{equation*}
\theta_2 (\alpha; \omega_1, \omega_2) \coloneqq \frac{1}{a_3} \sum_{n \in \alpha + \Z^2} ( a_2 n_1 + 2 a_3 n_2) n_1 e^{\frac{\pi i (a_2 n_1 + 2 a_3 n_2)^2 \omega_1}{2 a_3} +\frac{ \pi i D n_1^2 \omega_2}{2 a_3}} .
\end{equation*}

In particular, we note that if $\alpha \in \Z^2$, the term $n = (0,0)$ vanishes in each of the theta functions, and therefore so does $\mathcal{E}_{\alpha} (\tau)$ at $n = (0,0)$. We aim to show the following proposition. 

\begin{proposition}\label{Proposition: vector-valued QMF}
	The function $$\mathcal{E} (\tau) \coloneqq \sum_{ \alpha \in \J^*} \varepsilon (\alpha) \mathcal{E}_{ \alpha} (s \tau)$$ is a linear combination of components of a vector-valued quantum modular form of depth two and weight one for $\text{SL}_2 (\Z)$.
\end{proposition}

\begin{remark}
	Though we do not explore the situation here, for a fixed $\alpha$ the term $\mathcal{E}_{ \alpha} ( \tau)$ can itself be viewed as a modular form on a suitable subgroup $\Gamma$ of $\text{SL}_2 (\Z)$. As mentioned in Section \ref{Section: quantum set} a larger quantum set can be used here (if it is not already $\mathbb{Q}$), modifying the level of $\Gamma$ where appropriate.
\end{remark}

\begin{proof}[Proof of Proposition \ref{Proposition: vector-valued QMF}]
	We start by rewriting $\mathcal{E} (\tau)$ in terms of Shimura theta functions $\Theta_1 (A, h, N; \tau)$ - see Section \ref{Section: Shimura theta functions} for the relevant definitions.
	For $\theta_1$ set $\nu_1 = 2 a_1 n_1 + a_2 n_2 \in 2 a_1 \alpha_1 + a_2 \alpha_2 + \Z$ and $\nu_2 = n_2 \in \alpha_2 + \Z$. 
	We further have that $\nu_1 - a_2 \nu_2 = 2 a_1 n_1 \in 2 a_1 \alpha_1 + 2 a_1 \Z$ .
	
	Putting these into the definition we obtain that
	
	\begin{equation*}
	\theta_1 (\alpha; \omega_1, \omega_2) = \frac{1}{a_1} \sum_{\substack{\nu \in (2 a_1 \alpha_1 + a_2 \alpha_2, \alpha_2) + \Z^2 \\ \nu_1 - a_2 \nu_2 \in 2 a_1 \alpha_1 + 2 a_1 \Z}} \nu_1 \nu_2 e^{\frac{\pi i \nu_1^2 \omega_1}{2 a_1} + \frac{\pi i D \nu_2^2 \omega_2}{2 a_1}} .
	\end{equation*}
	We then rewrite $\theta_1 (\alpha; \omega_1, \omega_2)$ as
	
	\begin{equation*}
	\frac{1}{a_1} \sum_{\varrho \in \{ 0,1,\dots, 2a_1 -1\}} \left( \sum_{\nu_1 \in 2 a_1 \alpha_1 + a_2 \alpha_2 + a_2 \varrho + 2 a_1 \Z } {\nu_1} e^{\frac{\pi i \nu_1^2 \omega_1}{2 a_1}} \sum_{\nu_2 \in \alpha_2 + \varrho + 2 a_1 \Z }  \nu_2 e^{ \frac{ \pi i D \nu_2^2 \omega_2}{2 a_1}} \right) . 
	\end{equation*}
	Summing over $\alpha$ in the set $\J^*$ then gives
	
	\begin{equation*}
	\begin{split}
	\sum_{\alpha \in \J^*} \varepsilon(\alpha) \theta_1(\alpha ; \omega_1, \omega_2) & = \frac{1}{a_1 s^2} \sum_{A \in \mathcal{A}} \varepsilon_1(A) \sum_{\nu_1 \equiv A_1 \pmod{2 a_1 s}} \nu_1 e^{ \frac{\pi i \nu_1^2 \omega_1}{2  a_1 s^2}}  \sum_{\nu_2 \equiv A_2 \pmod{2 a_1 s}} \nu_2 e^{\frac{ \pi i D \nu_2^2 \omega_2}{2 a_1 s^2}} \\
	&= \frac{1}{a_1 s^2} \sum_{A \in \mathcal{A}} \varepsilon_1(A) \Theta_1 \left(2 a_1 s, A_1, 2 a_1 s; \frac{\omega_1}{s}\right) \Theta_1 \left(2 a_1 s, A_2, 2 a_1 s; \frac{ D \omega_2}{s}\right),
	\end{split}
	\end{equation*}
	where
	\begin{equation*}
	\mathcal{A} \coloneqq \{ ( 2 a_1 s \alpha_1 +  a_2 s \alpha_2 + a_2 \varrho s , s \alpha_2 + \varrho s) \mid \alpha \in \J^* , 0 \leq \varrho \leq 2a_1 - 1 \} \pmod{2 a_1 s}
	\end{equation*} 
	and $\varepsilon_1(A) \coloneqq \varepsilon (  \frac{A_1 - a_2 A_2}{2 a_1 s}, \frac{A_2}{s} )$. Note that $\mathcal{A}$ has size $2 a_1 N$, where we count elements with multiplicity.

	There is a similar situation for $\theta_2$, where we let  
	\begin{equation*}
	\mathcal{B} \coloneqq \{ (2 a_3 s \alpha_2 + a_2 s \alpha_1 + a_2 \varrho s , s \alpha_1  + \varrho s) \mid \alpha \in \J^* ,  0 \leq \varrho \leq 2a_3 - 1\} \pmod{2a_3 s }
	\end{equation*}
	of size $2a_3N$ along with $\varepsilon_2 (B) \coloneqq \varepsilon (\frac{B_2}{s}, \frac{B_1 - a_2 B_2}{2a_3s} )$. We obtain that $\mathcal{E} (\tau)$ is given by the expression
	
	\begin{equation*}
	\begin{split}
	& - \frac{\sqrt{D}}{4 a_1 s^2} \sum_{A \in \mathcal{A}} \varepsilon_1(A) \int_{- \bar{\tau}}^{i \infty} \int_{\omega_1}^{i \infty} \frac{ \Theta_1 (2 a_1 s, A_1, 2 a_1 s; \omega_1) \Theta_1 (2 a_1 s, A_2, 2 a_1 s; D \omega_2)}{\sqrt{-i (\omega_1 + \tau)} \sqrt{-i(\omega_2 + \tau)}} d\omega_2 d \omega_1 \\
	& - \frac{\sqrt{D}}{4 a_3 s^2} \sum_{B \in \mathcal{B}} \varepsilon_2 (B) \int_{- \bar{\tau}}^{i \infty} \int_{\omega_1}^{i \infty} \frac{ \Theta_1 \left(2 a_3 s, B_1, 2 a_3 s;\omega_1 \right) \Theta_1 \left(2 a_3 s, B_2, 2 a_3 s; D \omega_2 \right)}{{\sqrt{-i (\omega_1 + \tau)} \sqrt{-i(\omega_2 + \tau)}}} d\omega_2 d \omega_1 .
	\end{split}
	\end{equation*}
	
	For $n \in \N$, we note the equality
	
	\begin{equation}\label{Eqaution: Shimura theta tranform for vec valued}
	\begin{split}
	\Theta_1 (a, b, a; n \tau) = \sum_{j \in \Z} (aj + b)q^{\frac{n}{2a} (aj + b)^2} = & \frac{1}{n} \sum_{j \in \Z} (an j + bn) q^{\frac{1}{2an} (an j + bn)^2} = \frac{1}{n} \Theta_1 (na, nb, na; \tau) .
	\end{split}
	\end{equation}
	We split $\mathcal{E} (\tau) = \mathcal{E}_A (\tau) + \mathcal{E}_B (\tau)$ where
	
	\begin{equation*}
	\begin{split}
	\mathcal{E}_A (\tau) & \coloneqq - \frac{\sqrt{D}}{4 a_1 s^2} \sum_{A \in \mathcal{A}} \varepsilon_1(A) \int_{- \bar{\tau}}^{i \infty} \int_{\omega_1}^{i \infty} \frac{ \Theta_1 (2 a_1 s, A_1, 2 a_1 s; \omega_1) \Theta_1 (2 a_1 s, A_2, 2 a_1 s; D \omega_2)}{\sqrt{-i (\omega_1 + \tau)} \sqrt{-i(\omega_2 + \tau)}} d\omega_2 d \omega_1 , \\
	\mathcal{E}_B (\tau) & \coloneqq - \frac{\sqrt{D}}{4 a_3 s^2} \sum_{B \in \mathcal{B}} \varepsilon_2 (B) \int_{- \bar{\tau}}^{i \infty} \int_{\omega_1}^{i \infty} \frac{ \Theta_1 \left(2 a_3 s, B_1, 2 a_3 s;\omega_1 \right) \Theta_1 \left(2 a_3 s, B_2, 2 a_3 s; D \omega_2 \right)}{{\sqrt{-i (\omega_1 + \tau)} \sqrt{-i(\omega_2 + \tau)}}} d\omega_2 d \omega_1 .
	\end{split}
	\end{equation*}
	We concentrate firstly on $\mathcal{E}_A (\tau)$ and, for $k_1 \pmod{2a_1s}$ and $k_2 \pmod{2 D a_1s}$, set 
	
	\begin{equation*}
	I_{k_1, k_2} (\tau) \coloneqq I_{\Theta_1 (2 a_1 s, k_1, 2 a_1 s; \cdot ) , \Theta_1 (2 D a_1 s,  D k_2 , 2 D a_1 s;  \cdot) }  (\tau).
	\end{equation*}
	Via \eqref{Equation: transformation of Shimura for vector valued} we compute the transformations of the two Shimura theta functions as
	
	\begin{equation*}
	\begin{split}
	\Theta_1 \left( 2 a_1 s, k_1, 2 a_1 s; - \frac{1}{\tau} \right) =  \frac{(-i) (-i \tau)^{\frac{3}{2}}}{\sqrt{2a_1 s}} \sum_{j \pmod{2a_1s}} e \left( \frac{j k_1}{2 a_1 s} \right) \Theta_1 (2 a_1 s, j, 2 a_1 s; \tau ) 
	\end{split}
	\end{equation*}
	and
	
	\begin{equation*}
	\begin{split}
	\Theta_1 \left(2 D a_1 s, D k_1 , 2 D a_1 s; - \frac{1}{\tau} \right) =\frac{(-i) (-i \tau)^{\frac{3}{2}}}{\sqrt{2 D a_1 s}} \sum_{j \pmod{2 D a_1s}} e \left( \frac{j k_2}{2 a_1 s} \right) \Theta_1 (2 D a_1 s, j, 2 D a_1 s; \tau ) .
	\end{split}
	\end{equation*}
	Using \eqref{Eqaution: Shimura theta tranform for vec valued} we find
	
	\begin{equation*}
	\mathcal{E}_A (\tau) = - \frac{1}{4 a_1 s^2 \sqrt{D}} \sum_{\alpha \in \J^*} \sum_{A \in \mathcal{A}_\alpha} \varepsilon_1(A) I_{A_1, A_2} (\tau)  ,
	\end{equation*}
	where for a fixed $\alpha \in \J^*$ we define
	\begin{equation*}
	\mathcal{A}_\alpha \coloneqq \{ ( 2 a_1 s \alpha_1 +  a_2 s \alpha_2 + a_2 \varrho s , s \alpha_2 + \varrho s) \mid  0 \leq \varrho \leq 2a_1 - 1 \} \pmod{2 a_1 s}.
	\end{equation*}
	
	Then using Proposition \ref{Proposition: transformation for vector-valued} we obtain the transformation formula
	
	\begin{equation*}
	\begin{split}
	& \sum_{\alpha \in \J^*} \sum_{A \in \mathcal{A}_\alpha} \varepsilon_1(A) I_{A_1, A_2} (\tau) \\
	&- \frac{(- i \tau)^{-1}}{2 a_1 s \sqrt{D}} \sum_{\alpha \in \J^*} \sum_{A \in \mathcal{A}_\alpha} \varepsilon_1(A) \sum_{\substack{k_1 \pmod{2a_1 s} \\ k_2 \pmod{2 D a_1 s}}} e \left( \frac{k_1 A_1 + k_2 A_2}{2 a_1 s} \right) I_{k_1, \frac{k_2}{D}} \left(- \frac{1}{\tau} \right) \\
	= & \sum_{\alpha \in \J^*} \sum_{A \in \mathcal{A}_\alpha}  \varepsilon_1(A) \Bigg ( \int_0^{i \infty} \int_{\omega_1}^{i \infty} \frac{\Theta_1 (2 a_1 s, A_1, 2 a_1 s; \omega_1) \Theta_1 ( 2 D a_1 s, D A_2 ,  2 D a_1 s; \omega_2)}{\sqrt{-i(\omega_1 + \tau)} \sqrt{-i(\omega_2 + \tau)} } d\omega_1 d\omega_2 \\
	& +  I_{\Theta_1 (2 a_1 s, A_1 , 2 a_1 s; \cdot)} (\tau) r_{\Theta_1 ( 2 D a_1 s, D A_2 ,  2 D a_1 s; \cdot)} (\tau) - r_{\Theta_1 (2 a_1 s, A_1, 2 a_1 s; \cdot)} (\tau) r_{\Theta_1 ( 2 D a_1 s, D A_2 ,  2 D a_1 s; \cdot)} (\tau) \Bigg )
	\end{split}
	\end{equation*}
	
	Choosing $(k_1, k_2) = (A_1, D A_2)$ in the second term then returns our original Eichler integral. Each choice of $A \in \mathcal{A}_\alpha$ is then seen to be a component of a vector valued quantum modular form.
	In cases where $e(A) \coloneqq e \left( \frac{A_1^2 + D A_2^2}{2 a_1 s} \right)$ is the same across choices of $A \in \mathcal{A}_\alpha$, one can take this outside of the sum on $A$ as a constant factor, and so $ \sum_{A \in \mathcal{A}_\alpha} \varepsilon_1(A) I_{A_1, A_2} (\tau)$ can be seen as a single component of a vector-valued quantum modular form.
	Furthermore, if $e(A)$ is also constant across choices of $\alpha \in \J^*$ then we view all of $\mathcal{E}_A (\tau)$ as a single component.
	
	A similar statement holds for $\mathcal{E}_B$, and then one can easily put all components into a single vector-valued form in the obvious way. \qedhere
\end{proof}

\begin{example}\textit{(continued)}
	Returning to our example we see that we set
	\begin{equation*}
	\theta_1 (\alpha; \omega_1, \omega_2) \coloneqq \frac{1}{2} \sum_{n \in \alpha + \Z^2} (4 n_1 + n_2) n_2 e^{\frac{\pi i (4 n_1 + n_2)^2 \omega_1}{4} + \frac{ 7 \pi i  n_2^2 \omega_2}{4}}
	\end{equation*}
	along with the similar expression for $\theta_2$. Working through, we set $\nu_1 = 4n_1 + n_2$ and $\nu_2 = n_2$ so that $\nu_1 - \nu_2 = 4n_1$, giving the expression in terms of Shimura theta functions as
	
	\begin{equation*}
	\sum_{\alpha \in \J^*} \varepsilon(\alpha) \theta_1(\alpha ; \omega_1, \omega_2) =  \frac{1}{32} \sum_{A \in \mathcal{A}} \varepsilon_1(A) \Theta_1 \left(16, A_1, 16; \frac{\omega_1}{4}\right) \Theta_1 \left(16, A_2, 16; \frac{ 7 \omega_2}{4}\right),
	\end{equation*}
	where, after a little calculation, we have the set 
	
	\begin{equation*}
	\mathcal{A} = \left\{ (5,1) , (6,2) , (9,5) , (10,6) , (13,9) , (14,10) , (1,13) , (2,14)  \right\} \pmod{16}.
	\end{equation*}
	Further, we set $\varepsilon_1 (A) = \varepsilon(  \frac{A_1 - A_2}{16}, \frac{A_2}{4} )$. It is then simple to check that $e(A)$ is constant across the set $\mathcal{A}$, and hence $\mathcal{E}_A$ is a single component. We also find that the similarly defined function $e(B)$ is constant across the set
	
	\begin{equation*}
	\mathcal{B} = \left\{ (5,1), (3,1) , (1,5), (7,5) \right\} \pmod{8}.
	\end{equation*}
	Hence we view our Eichler integral as a single component of the vector-valued form.

\end{example}

\section{Indefinite Theta Functions}\label{Section: Indefinite theta functions 1}

Here we realise the double Eichler integrals as pieces of indefinite theta functions, with coefficients given by double error functions. We first write $\mathbb{E} (\tau) \coloneqq \mathcal{E} (\frac{\tau}{s})$ in such a way that we can apply the Euler-Maclaurin summation formula.

\begin{lemma}\label{E(tau) is sum of M_2}
	Let $u(n_1, n_2) \coloneqq (u_1 , u_2) = (\sqrt{v} (2 \sqrt{a_1} n_1 + \frac{a_2}{\sqrt{a_1}} n_2), \sqrt{v} m n_2)$, with $m \coloneqq \sqrt{4 a_3 - \frac{a_2^2}{a_1}}$, and $\kappa \coloneqq \frac{a_2}{m \sqrt{a_1}} = \frac{a_2}{\sqrt{D}}$ . We have that
	\begin{equation*}
	\mathbb{E} (\tau) = \frac{1}{2} \sum_{\alpha \in \J^*} \varepsilon (\alpha) \sum_{n \in \alpha + \Z^2 } M_2(\kappa; u) q^{-Q(n)} .
	\end{equation*}
\end{lemma}

\begin{proof}
	The claim follows once we have shown that 
	\begin{equation*}
	\begin{split}
	M_2 (\kappa; u)  = & - \frac{\sqrt{D} n_2 (2 a_1 n_1 + a_2 n_2)}{2 a_1} q^{Q(n)} \int_{- \bar{\tau}}^{i \infty} \frac{ e^{\frac{\pi i (2 a_1 n_1 + a_2 n_2)^2 \omega_1}{2 a_1}}}{\sqrt{-i (\omega_1 + \tau)}} \int_{\omega_1}^{i \infty} \frac{e^{\frac{\pi i D n_2^2 \omega_2}{2 a_1}}}{\sqrt{-i (\omega_2 + \tau)}} d\omega_2 d\omega_1\\
	& - \frac{\sqrt{D} n_1 (a_2 n_1 + 2 a_3 n_2)}{2 a_3} q^{Q(n)}  \int_{- \bar{\tau}}^{i \infty} \frac{e^{\frac{\pi i (a_2 n_1 + 2 a_3 n_2)^2 \omega_1}{2 a_3}}}{\sqrt{-i (\omega_1 + \tau)}} \int_{\omega_1}^{i \infty} \frac{e^{\frac{\pi i D n_1^2 \omega_2}{2 a_3}}}{\sqrt{-i (\omega_2 + \tau)}} d\omega_2 d\omega_1 .
	\end{split}
	\end{equation*}
	There are three different cases to consider, since we do not have the term $n = (0,0)$:
	
	\begin{enumerate}
		\item Both $n_1 \neq 0$ and $n_2 \neq 0$.
		\item We have $n_1 = 0$ and $n_2 \neq 0 \iff u_1 - \kappa u_2 = 0$ and $u_2 \neq 0$.
		\item We have $n_1 \neq 0$ and $n_2 = 0 \iff u_1 - \kappa u_2 \neq 0$ and $u_2 = 0$.
	\end{enumerate}
	We argue as in \cite{Higher_depth_QMFs}, and for the first case obtain that
	
	\begin{equation*}
	\begin{split}
	& M_2 (\kappa; u) = - \frac{u_1}{2\sqrt{v}} \frac{u_2}{\sqrt{v}} q^{\frac{u_1^2}{4v} + \frac{u_2^2}{4v}}  \int_{- \bar{\tau}}^{i \infty} \frac{e^ {\frac{\pi i u_1^2 \omega_1}{2v}}}{\sqrt{-i(\omega_1 + \tau)}} \int_{\omega_1}^{i \infty}  \frac{e^ {\frac{\pi i u_2^2 \omega_2}{2v}}}{\sqrt{-i(\omega_2 + \tau)}} d\omega_2 d\omega_1 \\
	& - \frac{u_1 - \kappa u_2}{2\sqrt{ (1+\kappa^2) v}} \frac{u_2 + \kappa u_1}{\sqrt{ (1+\kappa^2) v}} q^{\frac{(u_2 + \kappa u_1)^2}{4 (1+\kappa^2)v} + \frac{(u_1 - \kappa u_2)^2}{4(1+\kappa^2)v}}  \int_{- \bar{\tau}}^{i \infty} \frac{e^ {\frac{\pi i (u_2 + \kappa u_1)^2 \omega_1}{2 (1+\kappa^2)v}}}{\sqrt{-i(\omega_1 + \tau)}} \int_{\omega_1}^{i \infty}  \frac{e^ {\frac{\pi i (u_1 - \kappa u_2)^2 \omega_2}{2(1+\kappa^2)v}}}{\sqrt{-i(\omega_2 + \tau)}} d\omega_2 d\omega_1 .
	\end{split}
	\end{equation*}
	Plugging in the definitions of $u$ and $\kappa$ here yields the result directly.
	
	For case 2 we set $f_1 (v) \coloneqq M_2 ( \kappa; \frac{a_2}{\sqrt{a_1}} \sqrt{v} n_2 , m \sqrt{v} n_2 )$ and we want to prove the equality 
	
	\begin{equation*}
	f_1 (v) = - \frac{\sqrt{D}  a_2 n^2_2}{2 a_1} e^{2 \pi i a_3 n_2^2 \tau} \int_{- \bar{\tau}}^{i \infty} \frac{ e^{\frac{\pi i (a_2 n_2)^2 \omega_1}{2 a_1}}}{\sqrt{-i (\omega_1 + \tau)}} \int_{\omega_1}^{i \infty} \frac{e^{\frac{ \pi i D n_2^2 \omega_2}{2 a_1}}}{\sqrt{-i (\omega_2 + \tau)}} d\omega_2 d\omega_1 .
	\end{equation*}
	Letting $\omega_1 \mapsto \omega_1 - \tau$ and $\omega_2 \mapsto \omega_2 - \tau$ where $\tau = u + iv$ the right-hand side becomes
	
	\begin{equation*} 
	\begin{split}
	&  - \frac{\sqrt{D} a_2}{2 a_1} n_2^2 \int_{2 i v}^{i \infty} \frac{ e^{\frac{\pi i (a_2 n_2)^2 \omega_1}{2 a_1}}}{\sqrt{-i \omega_1 }} \int_{\omega_1}^{i \infty} \frac{e^{\frac{ \pi i D n_2^2 \omega_2}{2 a_1}}}{\sqrt{-i \omega_2}} d\omega_2 d\omega_1 \\
	=&  	\frac{\sqrt{D} a_2}{a_1} n_2^2 \int_{v}^{\infty} \frac{ e^{\frac{- \pi (a_2 n_2)^2 \omega_1}{a_1}}}{\sqrt{\omega_1 }} \int_{\omega_1}^{\infty} \frac{e^{\frac{-\pi D n_2^2 \omega_2}{a_1}}}{\sqrt{\omega_2}} d\omega_2 d\omega_1 \eqqcolon f_2(v) .
	\end{split}
	\end{equation*}
	By \eqref{Equation: relation between M_2 and E_2} we have that 
	
	\begin{equation*}
	f_1(v) = E_2 \left( \kappa; \frac{a_2}{\sqrt{a_1}} \sqrt{v} n_2 , m \sqrt{v} n_2 \right) - \sgn(n_2) E_1 \left(\frac{a_2}{\sqrt{a_1}} \sqrt{v} n_2 \right).
	\end{equation*}
	Considering differentials in $v$ we obtain
	
	\begin{equation*}
	\begin{split}
	f_1 '(v) = & \frac{n_2}{2 \sqrt{v}} \left( \frac{a_2}{\sqrt{a_1}} E_2^{(1,0)} \left(\kappa; \frac{a_2}{\sqrt{a_1}} \sqrt{v} n_2 , m \sqrt{v} n_2 \right) + m E_2^{(0,1)} \left(\kappa; \frac{a_2}{\sqrt{a_1}} \sqrt{v} n_2 , m \sqrt{v} n_2\right) \right)  \\
	& -   \frac{n_2}{2 \sqrt{v}} \sgn(n_2) \frac{a_2}{\sqrt{a_1}} E_1'\left(\frac{a_2}{\sqrt{a_1}} \sqrt{v} n_2\right) \\
	= & \frac{n_2}{2 \sqrt{v}} \left( \frac{2 a_2}{\sqrt{a_1}} e^{\frac{- \pi a_2^2 v n_2^2}{a_1}} E\left(m \sqrt{v} n_2 \right) + \frac{2 (\kappa + 1)}{\sqrt{1+ \kappa^2}} e^{\frac{- \pi \left(m \sqrt{v} n_2 + \frac{\kappa a_2}{\sqrt{a_1}} \sqrt{v} n_2 \right)^2}{1+ \kappa^2}} E(0)  - \sgn(n_2) \frac{a_2}{\sqrt{a_1}} 2 e^{\frac{- \pi a_2^2 v n_2^2}{a_1}} \right) \\
	= & \frac{a_2 n_2}{\sqrt{v a_1}}  e^{\frac{- \pi a_2^2 v n_2^2}{a_1}} \left(  E\left(m \sqrt{v} n_2 \right) - \sgn(n_2) \right) .
	\end{split}
	\end{equation*}
	Since $m > 0$ we have  
	\begin{equation*}
	E\left(m \sqrt{v} n_2 \right) - \sgn(n_2) = M\left(m \sqrt{v} n_2 \right) = \frac{- \sgn(n_2)}{\sqrt{\pi}} \Gamma \left(\frac{1}{2} , \pi m^2 v n_2^2\right),
	\end{equation*} 
	using \eqref{Equation: M(u) in terms of incomplete gamma functions}. Thus we obtain
	
	\begin{equation*}
	f_1 '(v) =  - \frac{a_2 |n_2 |}{\sqrt{v a_1 \pi }}  e^{\frac{- \pi a_2^2 v n_2^2}{a_1}} \Gamma \left(\frac{1}{2} , \pi m^2 v n_2^2 \right) .
	\end{equation*}
	We then consider the differential of $f_2(v)$. Computing directly we obtain
	
	\begin{equation*}
	\begin{split}
	f_2 ' (v) & =  \frac{ \sqrt{D} a_2}{a_1} n_2^2 (-1) \frac{e^{\frac{- \pi (a_2 n_2)^2 v}{a_1}}}{\sqrt{v}} \int_v^{\infty} e^{\frac{- \pi (D n_2^2) \omega_2}{a_1}} \sqrt{\omega_2} \frac{d \omega_2}{\omega_2} \\
	& = \frac{ \sqrt{D} a_2}{a_1} n_2^2 (-1) \frac{e^{\frac{- \pi (a_2 n_2)^2 v}{a_1}}}{\sqrt{v}} \frac{ \sqrt{a_1}}{\sqrt{\pi D n_2^2}} \Gamma \left( \frac{1}{2} , m^2 \pi v n_2^2 \right) = \frac{- a_2 |n_2|}{\sqrt{a_1 v \pi}} e^{ \frac{- \pi (a_2 n_2)^2 v}{a_1}} \Gamma \left( \frac{1}{2} , m^2 \pi v n_2^2 \right) .
	\end{split}
	\end{equation*}
	Setting $f (v) = f_1(v) - f_2(v)$ we see that $f'(v) = 0$, and since $\lim_{v \rightarrow \infty} f(v) = 0$ we obtain that $f_1 (v) = f_2 (v)$ as required.
	
	For case 3 a similar argument holds, setting $f_3 (v) \coloneqq M_2 ( \kappa; 2 \sqrt{a_1} \sqrt{v} n_1 , 0)$. The claim now follows. \qedhere	
\end{proof}

\section{Asymptotic behaviour of the double Eichler integral}
In this section we relate the functions $\mathbb{E}$ and $F$. Letting $F (e^{2 \pi i \frac{h}{k} - t}) \eqqcolon \sum_{m \geq 0} a_{h,k} (m) t^m$ as $t \rightarrow 0^+$, we prove the following Theorem.

\begin{theorem}\label{Theorem: asymptotics agree}
	For $h,k \in \mathcal{Q}$ as determined by Section \ref{Section: quantum set} we have that 
	\begin{equation*}
	\mathbb{E} \left(\frac{h}{k} + \frac{it}{2 \pi} \right) \sim \sum_{m \geq 0} a_{-h,k} (m) (-t)^m .
	\end{equation*}
\end{theorem} 

\begin{proof}
	Using Lemma \ref{E(tau) is sum of M_2} and that $M_2$ is an even function we have 
	
	\begin{equation*}
	\begin{split}
	\mathbb{E} (\tau) = & \frac{1}{2} \sum_{\alpha \in \J} \varepsilon (\alpha) \sum_{n \in \alpha + \mathbb{N}_0^2} M_2 (\kappa; u(n_1, n_2)) q^{-Q(n_1, n_2)} \\
	& + \frac{1}{2} \sum_{\alpha \in \widetilde{\J}} \widetilde{\varepsilon} (\alpha) \sum_{n \in \alpha + \mathbb{N}_0^2} M_2 (\kappa; u(-n_1, n_2)) q^{-Q(-n_1, n_2)} ,
	\end{split}
	\end{equation*}
	with $\widetilde{\mathcal{J}} \coloneqq \{ (1 - \alpha_1 , \alpha_2) \mid \alpha \in \J \} $ and $\widetilde{\varepsilon} (\alpha_1, \alpha_2) \coloneqq \varepsilon(1-\alpha_1, \alpha_2)$.
	
	In order to be able to apply the Euler-Maclaurin summation formula, we define $M_2^* (\kappa; x_1, x_2)$ by replacing each $\sgn$ with $\sgn^*$. Explicitly, we set
	
	\begin{equation}\label{Equation: definition of M_2^*}
	\begin{split}
	M_2^* (\kappa; u_1, u_2) \coloneqq & \sgn^*(x_1)\sgn^*(x_2) + E_2(\kappa; x_1 + k x_2, x_2) - \sgn^*(x_2) E(x_1 + \kappa x_2) \\
	& - \sgn^*(x_1) E \left( \frac{\kappa x_1}{\sqrt{1+\kappa^2}} + \sqrt{1+\kappa^2} x_2 \right).
	\end{split}
	\end{equation}
	It is easy to see that, using \eqref{Equation: relation between M_2 and E_2} and \eqref{Equation: definition of M_2^*}, we have
	\begin{equation*}
	\begin{split}
	& M_2 (\kappa; u_1(0,x_2), u_2(x_2))  - \lim_{x_1 \rightarrow 0^+} M_2^* (\kappa; u_1(\pm x_1, x_2), u_2(x_2)) = \pm M\left(\sqrt{1+\kappa^2} x_2\right) , \\
	& M_2 (\kappa; u_1(x_1,0), u_2(0))  - \lim_{x_2 \rightarrow 0^+} M_2^* (\kappa; u_1(\pm x_1, x_2), u_2(x_2)) = \pm M( x_1 ) .
	\end{split}
	\end{equation*}

	We then rewrite $\mathbb{E} (\tau) =  \mathcal{E}^* (\tau) + H_1 (\tau) + H_2 (\tau)$, defining
	
	\begin{equation*}
	\begin{split}
	\mathcal{E}^* (\tau) \coloneqq & \frac{1}{2} \sum_{\alpha \in \J} \varepsilon (\alpha) \sum_{n \in \alpha + \mathbb{N}_0^2} M_2^* (\kappa; u(n_1, n_2)) q^{-Q(n_1, n_2)} \\
	& + \frac{1}{2} \sum_{\alpha \in \widetilde{\J}} \widetilde{\varepsilon} (\alpha) \sum_{n \in \alpha + \mathbb{N}_0^2} M_2^* (\kappa; u(-n_1, n_2)) q^{-Q(-n_1, n_2)} ,
	\end{split}
	\end{equation*}
	along with the boundary terms 
	
	\begin{equation*}
	\begin{split}
	H_1 (\tau) \coloneqq - &\frac{1}{2} \sum_{\alpha \in \J^*_1} \varepsilon(\alpha) \sgn^* (\alpha_1) \\
	\times &\left( \sum_{j \in \alpha_2 + \mathbb{N}_0 } M \left( j  \sqrt{(1+\kappa^2) v} \right) q^{-a_3 j^2} -  \ \sum_{j \in 1 - \alpha_2 + \mathbb{N}_0 }  M \left( j \sqrt{(1+\kappa^2) v} \right) q^{-a_3 j^2}  \right) 
	\end{split}
	\end{equation*}
	and
	
	\begin{equation*}
	H_2 (\tau) \coloneqq - \frac{1}{2} \sum_{\alpha \in \J^*_2} \varepsilon(\alpha) \sgn^* (\alpha_2) \left( \sum_{j \in \alpha_1 + \mathbb{N}_0 } M \left( j  \sqrt{v} \right) q^{-a_1 j^2} -  \sum_{j \in 1 - \alpha_1 + \mathbb{N}_0 }  M \left( j \sqrt{v} \right) q^{-a_1 j^2} \right) .
	\end{equation*}
	If $\alpha_1 \in \Z$ (resp. $\alpha_2 \in \Z$) then for the $n_1 = 0$ (resp. $n_2 = 0$) we take the limit $n_1 \rightarrow 0$ (resp. $n_2 \rightarrow 0$) in the $M_2^*$ functions. 
	
	\begin{remark}
		In the case that for every $(a, x)$  in $\J^*_1$, the element $(b, 1-x)$ also exists in $\J^*_1$, along with the conditions $\sgn^*(a) = \sgn^*(b)$ and $\varepsilon(a, x) = \varepsilon(b, 1-x)$, then $H_1 = 0$ identically. A similar statement holds for the function $H_2$.
	\end{remark}
	
	Using techniques similar to those in Section \ref{Section: asymptotic behaviour of F at Certain Roots of Unity} we next determine the asymptotic behaviour of $\mathcal{E}^* , H_1 $ and $H_2$. First we rewrite $\mathcal{E}^*$ as
	
	\begin{equation*}
	\begin{split}
	\mathcal{E}^* \left(\frac{h}{k} + \frac{it}{2 \pi} \right) & = \sum_{\alpha \in \J} \varepsilon (\alpha) \sum_{0 \leq \ell \leq \frac{ks}{\delta} -1} e^{-2 \pi i \frac{h}{k} Q(\ell_1 + \alpha_1, \ell_2 + \alpha_2) } \sum_{n \in \frac{\delta (\ell + \alpha)}{ks} + \mathbb{N}_0^2} \mathcal{F}_4 \left( \frac{ks}{\delta}\sqrt{t} n \right) \\
	& + \sum_{\alpha \in \widetilde{\J}} \widetilde{\varepsilon} (\alpha) \sum_{0 \leq \ell \leq \frac{ks}{\delta} -1} e^{-2 \pi i \frac{h}{k} Q(-(\ell_1 + \alpha_1), \ell_2 + \alpha_2) } \sum_{n \in \frac{\delta (\ell + \alpha)}{ks} + \mathbb{N}_0^2} \widetilde{\mathcal{F}_4} \left( \frac{ks}{\delta}\sqrt{t} n \right) ,
	\end{split}
	\end{equation*}
	with $\mathcal{F}_4 (x) \coloneqq \frac{1}{2} M_2^* (\kappa; \frac{1}{\sqrt{2 \pi}}( u (x_1, x_2))) e^{Q(x)}$ and $\widetilde{\mathcal{F}}_4 (x) \coloneqq \mathcal{F}_4 (u(-x_1, x_2))$.
	
	Then the contribution from the $\mathcal{F}_4$ term to the main term in the Euler-Maclaurin summation formula is given by
	
	\begin{equation*}
	\frac{\delta^2}{k^2 s^2 t} \mathcal{I}_{\mathcal{F}_4} \sum_{\alpha \in \J} \varepsilon (\alpha) \sum_{0 \leq \ell \leq \frac{ks}{\delta} -1} e^{-2 \pi i \frac{h}{k} Q(\ell+ \alpha) } ,
	\end{equation*}
	which vanishes, conjugating a result from Section \ref{Section: asymptotic behaviour of F at Certain Roots of Unity}. Similarly, the contribution from the $\widetilde{\mathcal{F}}_4$ to the main term of the Euler-Maclaurin summation formula also vanishes.
	
	The second term of \eqref{Euler-Maclaurin summation formula} is (again noting as in Section \ref{Section: asymptotic behaviour of F at Certain Roots of Unity} that terms where $n_2$ is even vanish)
	
	\begin{equation*}
	\begin{split}
	-2 \sum_{\alpha \in \J^*} \varepsilon (\alpha) & \sum_{0 \leq \ell \leq \frac{ks}{\delta} -1} e^{-2 \pi i \frac{h}{k} Q(\ell+ \alpha)} \sum_{n_2 \geq 0} \frac{ B_{2 n_2 + 2} \left( \frac{ \delta (\ell_2 + \alpha_2)}{ks}\right) }{(2n_2 + 2)!} \\
	& \times \int_{0}^\infty \left( \mathcal{F}_4^{(0, 2 n_2 +1)} (x_1, 0) + \widetilde{\mathcal{F}}_4^{(0, 2 n_2 +1)} (x_1, 0) \right) dx_1 \left( \frac{k^2 s^2 t}{\delta^2} \right) ^{n_2} .
	\end{split}
	\end{equation*} 
	We now claim that
	
	\begin{equation}\label{F_1equalsF_4}
	\int_0^\infty \left( \mathcal{F}_4^{(0, 2n_2 + 1)} (x_1, 0) + \widetilde{\mathcal{F}_4}^{(0,2n_2+1)} (x_1,0) \right) dx_1 = (-1)^{n_2} \int_0^\infty \mathcal{F}_1^{(0, 2n_2 + 1)} (x_1,0) dx_1,
	\end{equation}
	corresponding to the terms arising in equation \eqref{Equation: second term of asymptotics}. First, we simplify the right-hand side of \eqref{F_1equalsF_4}
	
	\begin{equation*}\begin{split}
	(-1)^{n_2} \int_0^\infty \mathcal{F}_1^{(0, 2n_2 + 1)} (x_1,0) dx_1 & = \left[ \frac{\partial^{2n_2+1}}{\partial x_2^{2n_2 + 1}} \int_0^\infty \mathcal{F}_1 (x_1, x_2) dx_1 \right]_{x_2= 0} \\
	&= \left[ \frac{\partial^{2n_2+1}}{\partial x_2^{2n_2 + 1}} e^{\frac{-m^2 x_2^2}{4}} \int_0^\infty e^{- \left(\sqrt{a_1} x_1 + \frac{a_2 x_2}{2 \sqrt{a_1}}\right)^2} dx_1 \right]_{x_2= 0} .
	\end{split}
	\end{equation*}
	Taking the integral without differentiating and substituting $\omega =\frac{1}{\sqrt{\pi}} \left( \sqrt{a_1}x_1 + \frac{a_2 x_2}{2 \sqrt{a_1}} \right)$ we get
	
	\begin{equation*}\begin{split}
	\int_0^\infty e^{- \left(\sqrt{a_1} x_1 + \frac{a_2 x_2}{2 \sqrt{a_1}} \right)^2} dx_1 &= \sqrt{\frac{\pi}{a_1}} \int_{\frac{a_2 x_2}{2\sqrt{a_1 \pi}}}^\infty e^{-\pi \omega^2} d \omega = \frac{\sqrt{\pi}}{2 \sqrt{a_1}} \left( 1 - E\left( \frac{a_2 x_2}{2 \sqrt{a_1 \pi}} \right) \right) .
	\end{split}
	\end{equation*}
	Therefore the right-hand side of \eqref{F_1equalsF_4} is given by
	
	\begin{equation*}
	\begin{split}
	 \left[ \frac{\sqrt{\pi}}{2 \sqrt{a_1}} \frac{\partial^{2n_2+1}}{\partial x_2^{2n_2 + 1}} e^{\frac{-m^2 x_2^2}{4}} \left( 1 - E\left( \frac{a_2 x_2}{2 \sqrt{a_1 \pi}} \right) \right) \right]_{x_2 =0}	= - \left[ \frac{\sqrt{\pi}}{2 \sqrt{a_1}} \frac{\partial^{2n_2+1}}{\partial x_2^{2n_2 + 1}} e^{\frac{-m^2 x_2^2}{4}} E\left( \frac{a_2 x_2}{2 \sqrt{a_1 \pi}} \right) \right]_{x_2 =0} ,
	\end{split}
	\end{equation*}
	since the other terms vanish under differentiation and setting $x_2 =0$.
	
	Next we concentrate on the left-hand side of \eqref{F_1equalsF_4}, and to ease notation we set
	
	\begin{equation*}
	\begin{split}
	& h_1 (x_1, x_2) \coloneqq E_2 (\kappa; u(x_1, x_2)) ,\\
	& h_2 (x_1, x_2) \coloneqq E(u_1(x_1,x_2)) ,\\
	& h_3 (x_1, x_2) \coloneqq E \left(\frac{\kappa x_1}{\sqrt{1+\kappa^2}} + \sqrt{1+\kappa^2}x_2 \right) .
	\end{split}
	\end{equation*}
	We also define 
	\begin{equation*}
	\begin{split}
	& c_0 (x_1, x_2) \coloneqq e^{Q(x_1, x_2)} , \\
	& c_j (x_1, x_2) \coloneqq h_j \left( \frac{1}{\sqrt{2 \pi}} (x_1, x_2) \right) e^{Q(x_1,x_2)} ,\\
	\end{split}
	\end{equation*}
	for $j =1,2,3$.
	
	By definition of $M_2^* (\kappa; u)$ we compute that 
	
	\begin{equation*}
	\begin{split}
	& \mathcal{F}_4^{(0,2n_2+1)} (x_1,0) + \widetilde{\mathcal{F}}_4^{(0,2n_2+1)} (x_1,0)\\
	= & \frac{1}{2} (c_0^{(0,2n_2+1)} (x_1,0) + c_1^{(0,2n_2+1)} (x_1,0) - c_2^{(0,2n_2+1)} (x_1,0) - c_3^{(0,2n_2+1)} (x_1,0)) \\
	& + \frac{1}{2} (- c_0^{(0,2n_2+1)} (-x_1,0) + c_1^{(0,2n_2+1)} (-x_1,0) - c_2^{(0,2n_2+1)} (-x_1,0) + c_3^{(0,2n_2+1)} (-x_1,0)) \\
	= & c_0^{(0,2n_2+1)} (x_1,0) - c_2^{(0,2n_2+1)} (x_1,0) ,
	\end{split}
	\end{equation*}
	using that $c_0$ and $c_1$ are even, whereas $c_2$ and $c_3$ are odd.
	
	Then we are considering the expression
	
	\begin{equation*}
	\begin{split}
	& - \frac{\partial^{2n_2+1}}{\partial x_2^{2n_2 + 1}} \left[ \int_{0}^\infty  \left( e^{Q(x_1,x_2)} - e^{Q(x_1,x_2)} E \left(\frac{1}{\sqrt{2 \pi}} u_1 (x_1, x_2) \right) \right) dx_1 \right]_{x_2=0} \\
	= & - \frac{\partial^{2n_2+1}}{\partial x_2^{2n_2 + 1}} \left[ \int_{0}^\infty   e^{Q(x_1,x_2)} M \left(\frac{1}{\sqrt{2 \pi}} u_1 (x_1, x_2)\right) dx_1  \right]_{x_2 =0} \\
	= & - \frac{\partial^{2n_2+1}}{\partial x_2^{2n_2 + 1}} \left[ e^{a_3 x_2^2}   \int_{0}^\infty e^{a_1 x_1^2 + a_2 x_1x_2} M\left(\frac{1}{\sqrt{2 \pi}} u_1 (x_1, x_2) \right) dx_1 \right]_{x_2 =0} .
	\end{split}
	\end{equation*}
	Taking the integral without differentiating, and letting $\omega = \frac{1}{\sqrt{2 \pi}} u_1(x_1, x_2)$ we obtain
	
	\begin{equation*}
	\begin{split}
	& -e^{\frac{m^2 x_2^2}{4}} \int_{ \frac{a_2 x_2 }{\sqrt{2 \pi a_1}}}^\infty M(\omega) e^{\frac{\pi \omega^2}{2}} \sqrt{ \frac{\pi}{2 a_1}} d\omega \\
	= & - \sqrt{ \frac{\pi}{2 a_1}} e^{\frac{m^2 x_2^2}{4}} \left( \int_0^\infty M(\omega) e^{\frac{\pi \omega^2}{2}} d \omega - \int_0^{ \frac{a_2 x_2 }{\sqrt{2 \pi a_1}}} M(\omega) e^{\frac{\pi \omega^2}{2}} d \omega \right) .
	\end{split}
	\end{equation*}
	After differentiating an odd number of times and evaluating at $x_2 =0 $ the terms arising from the first intgeral here vanish, and we decompose the second integral using $M(\omega) =  E(\omega) - 1$. Since $E(\omega)$ is odd, the contributions from this term also vanish, and overall we are left with
	
	\begin{equation*}
	\begin{split}
	& -  \sqrt{ \frac{\pi}{2 a_1}} \frac{\partial^{2n_2+1}}{\partial x_2^{2n_2 + 1}} \left[ e^{\frac{m^2 x_2^2}{4}} \int_0^{ \frac{a_2 x_2 }{\sqrt{2 \pi a_1}}}  e^{\frac{\pi \omega^2}{2}} d \omega \right]_{x_2 =0} \\
	=& -  \sqrt{ \frac{\pi}{2 a_1}} i^{- 2n_2 - 1} \frac{\partial^{2n_2+1}}{\partial x_2^{2n_2 + 1}}  \left[ e^{- \frac{m^2 x_2^2}{4}} \int_0^{ \frac{a_2 x_2 i}{\sqrt{2 \pi a_1}}}  e^{\frac{\pi \omega^2}{2}} d \omega \right]_{x_2 =0}  .
	\end{split}
	\end{equation*}
	
	The integral in question is therefore given by
	
	\begin{equation*}
	\begin{split}
	i \sqrt{2} \int_0^{ \frac{a_2 x_2}{2 \sqrt{ \pi a_1}}} e^{- \pi \omega^2} d\omega =  \frac{i}{\sqrt{2}} E\left(\frac{a_2 x_2}{2 \sqrt{a_1 \pi}} \right) .
	\end{split}
	\end{equation*}
	Given this, it is easy to conclude that the left-hand side of \eqref{F_1equalsF_4} is given by
	
	\begin{equation*}
	- \sqrt{ \frac{\pi}{4 a_1}} (-1)^{n_2} \frac{\partial^{2n_2+1}}{\partial x_2^{2n_2 + 1}} \left[  e^{- \frac{m^2 x_2^2}{4}} E\left(\frac{a_2 x_2}{2 \sqrt{a_1 \pi}} \right) \right]_{x_2 =0} 
	\end{equation*}
	as claimed.
	
	The third term in the Euler-Maclaurin summation formula is given by
	
	\begin{equation*}
	\begin{split}
	-2 \sum_{\alpha \in \J^*} \varepsilon (\alpha) & \sum_{0 \leq \ell \leq \frac{ks}{\delta} -1} e^{-2 \pi i \frac{h}{k} Q(\ell+ \alpha)} \sum_{n_1 \geq 0} \frac{ B_{2 n_1 + 2} \left( \frac{ \delta (l_1 + \alpha_1)}{ks}\right) }{(2n_1 + 2)!} \\
	& \times \int_{0}^\infty \left( \mathcal{F}_4^{( 2 n_1 +1, 0)} (0, x_2) + \widetilde{\mathcal{F}}_4^{( 2 n_2 +1,0)} (0, x_2) \right) dx_2 \left( \frac{k^2 s^2 t}{\delta^2} \right) ^{n_1} ,
	\end{split}
	\end{equation*}
	again observing that the terms with even $n_1$ vanish. A similar argument to before gives us that this is equal to
	
	\begin{equation*}
	\begin{split}
	-2 \sum_{\alpha \in \J^*} \varepsilon (\alpha) \sum_{0 \leq \ell \leq \frac{ks}{\delta} -1} e^{-2 \pi i \frac{h}{k} Q(\ell+ \alpha)} & \sum_{n_1 \geq 0} \frac{ B_{2 n_1 + 2} \left( \frac{ \delta (\ell_1 + \alpha_1)}{ks}\right) }{(2n_1 + 2)!} \\
	& \times \int_{0}^\infty \mathcal{F}_1^{(2n_1+1,0)} (0, x_2) (-1)^{n_1} dx_2 \left( \frac{k^2 s^2 t}{\delta^2} \right) ^{n_1} .
	\end{split}
	\end{equation*}
	
	The final term in \eqref{Euler-Maclaurin summation formula} is 
	
	\begin{equation*}
	\begin{split}
	2  \sum_{\alpha \in \J^*} \varepsilon (\alpha) \sum_{0 \leq \ell \leq \frac{ks}{\delta} -1}  e^{-2 \pi i \frac{h}{k} Q(\ell+ \alpha)} & \sum_{\substack{n_1, n_2 \geq 0 \\ n_1 \equiv n_2 \mod{0}}} \frac{B_{n_1 +1} \left( \frac{ \delta (\ell_1 + \alpha_1)}{ks}\right)}{(n_1 + 1)!}  \frac{B_{n_2 +1} \left( \frac{ \delta (\ell_2 + \alpha_2)}{ks}\right)}{(n_2 + 1)!} \\
	& \times \left(  \mathcal{F}_4^{(n_1, n_2 )} (0, 0) - (-1)^{n_1} \widetilde{\mathcal{F}}_4^{(n_1,n_2)} (0,0) \right) \left( \frac{ks \sqrt{t}}{\delta} \right)^{n_1 + n_2} .
	\end{split}
	\end{equation*}
	Via a similar argument to the one in \cite{Higher_depth_QMFs}, we have that
	
	\begin{equation*}
	\mathcal{F}_4^{(n_1, n_2 )} (0, 0) - (-1)^{n_1} \widetilde{\mathcal{F}_4}^{(n_1,n_2)} (0,0) = i^{n_1 + n_2} \mathcal{F}_1^{(n_1, n_2)} (0,0).
	\end{equation*}
	We can see this by decomposing the left-hand side as
	
	\begin{equation*}
	\mathcal{F}_4^{(n_1, n_2 )} (0, 0) - (-1)^{n_1} \widetilde{\mathcal{F}_4}^{(n_1,n_2)} (0,0) = c_0^{(n_1,n_2)} (0,0) - c_3^{(n_1,n_2)} (0,0) .
	\end{equation*}
	Using $c_3(-x_1, -x_2) = - c_3 (x_1, x_2) $ we have that 
	
	\begin{equation*}
	c_3^{(n_1,n_2)} (0,0) = (-1)^{n_1 + n_2 + 1} c_3^{(n_1,n_2)} (0,0) .
	\end{equation*}
	Since we have only cases where $n_1 \equiv n_2 \pmod{2}$ the contribution from the $c_3$ terms vanishes. Thus we are left with
	
	\begin{equation*}
	c_0^{(n_1,n_2)} (0,0) = i^{n_1 + n_2} \left[ \frac{\partial^{n_1}}{\partial x_1^{n_1}} \frac{\partial^{n_2}}{\partial x_2^{n_2}} e^{- Q(x_1, x_2)}  \right]_{x_1 = x_2= 0} = i^{n_1 + n_2} \mathcal{F}_1 (0,0).
	\end{equation*}
	
	Now consider the asympototic behaviour of the functions $H_1$ and $H_2$. Set $\mathcal{F}_5 (x) \coloneqq M ( {\frac{\sqrt{2}}{\sqrt{\pi}}} x) e^{a_3 x^2}$ and note that
	
	\begin{equation*}
	\mathcal{F}_5^{2m} (0) = (-1)^{m+1} \left[ \frac{\partial^{2m}}{\partial x^{2m}} e^{-a_3 x^2} \right] = (-1)^{m+1} \mathcal{F}_2^{2m} (0) .
	\end{equation*}
	
	The contribution to the Euler-Maclaurin summation formula in one dimension \eqref{Euler-Maclaurin summation formula one dim} arising from the $H_1$ term is then given by
	
	\begin{equation*}
	\begin{split}
	\frac{1}{2} & \sum_{\alpha \in \J^*_1} \varepsilon(\alpha) \sgn^* (\alpha_1) \sum_{0\leq r \leq \frac{ks}{\delta} - 1} e^{\frac{2 \pi i h}{k} a_3 (r + (1 - \alpha_2))^2} \sum_{m \geq 0} \frac{ B_{2m+1} \left( \frac{\delta (r + (1 - \alpha_2))}{ks} \right)}{(2m+1)!} \mathcal{F}_5^{(2m)} (0) \left( \frac{k^2 s^2 t}{\delta^2} \right)^{m} \\
	& - \frac{1}{2} \sum_{\alpha \in \J^*_1}  \varepsilon(\alpha) \sgn^* (\alpha_1) \sum_{0\leq r \leq \frac{ks}{\delta} - 1} e^{\frac{2 \pi i h}{k} a_3 (r + \alpha_2)^2} \sum_{m \geq 0} \frac{ B_{2m+1} \left( \frac{\delta (r + \alpha_2)}{ks} \right)}{(2m+1)!} \mathcal{F}_5^{(2m)} (0) \left( \frac{k^2 s^2 t}{\delta^2} \right)^{m} .
	\end{split}
	\end{equation*}
	Similarly, for the $H_2$ term we obtain the contribution
	
	\begin{equation*}
	\begin{split}
	\frac{1}{2} & \sum_{\alpha \in \J^*_2} \varepsilon(\alpha) \sgn^* (\alpha_2) \sum_{0\leq r \leq \frac{ks}{\delta} - 1} e^{\frac{2 \pi i h}{k} a_1 (r + (1 - \alpha_1))^2} \sum_{m \geq 0} \frac{ B_{2m+1} \left( \frac{\delta (r + (1 - \alpha_1))}{ks} \right)}{(2m+1)!} \mathcal{F}_6^{(2m)} (0) \left( \frac{k^2 s^2 t}{\delta^2} \right)^{m} \\
	& - \frac{1}{2} \sum_{\alpha \in \J^*_2} \varepsilon(\alpha) \sgn^* (\alpha_2) \sum_{0\leq r \leq \frac{ks}{\delta} - 1} e^{\frac{2 \pi i h}{k} a_1 (r + \alpha_1)^2} \sum_{m \geq 0} \frac{ B_{2m+1} \left( \frac{\delta (r + \alpha_1)}{ks} \right)}{(2m+1)!} \mathcal{F}_6^{(2m)} (0) \left( \frac{k^2 s^2 t}{\delta^2} \right)^{m} ,
	\end{split}
	\end{equation*}
	where  $\mathcal{F}_6 (x) \coloneqq M ( \frac{\sqrt{2}}{\sqrt{\pi}} x ) e^{a_1 x^2}$. Noting that
	
	\begin{equation*}
	\mathcal{F}_6^{2m} (0) = (-1)^{m+1} \left[ \frac{\partial^{2m}}{\partial x^{2m}} e^{-a_1 x^2} \right] = (-1)^{m+1} \mathcal{F}_3^{2m} (0) .
	\end{equation*}
	gives the claim. \qedhere
	
\end{proof}

\section{Proof of Theorem 1.2}\label{Section: Proof of main theorem}
We are now ready to prove a refined version of Theorem \ref{Theorem: introduction version of main theorem}.

\begin{theorem}\label{Theorem: Main Theorem - Paper}
	Let $\mathcal{Q}$ be as in Section \ref{Section: quantum set}. The function $\widehat{F} \colon \mathcal{Q} \rightarrow \C$  defined by $\widehat{F} (\frac{h}{k}) \coloneqq F(e^{2 \pi i s \frac{h}{k}})$ is a sum of components of a vector-valued quantum modular form of depth two and weight one. \\
	Moreover, with $\mathcal{A}$ and $\mathcal{B}$ as in Section \ref{Section: Multiple Eichler Integrals of weight 1}, if the functions 
	\begin{equation*}
	e \left( \frac{A_1^2 + D A_2^2}{2 a_1 s} \right) \hspace{10pt} \text{ , } \hspace{10pt} e \left( \frac{B_1^2 + D B_2^2}{2 a_3 s} \right)
	\end{equation*}
	are constant across all choices of $A \in \mathcal{A}$ and $B \in \mathcal{B}$ then $\widehat{F} (\frac{h}{k})$ is a single component of a vector-valued quantum modular form of depth two and weight one.
\end{theorem}

\begin{proof}
	By Theorem \ref{Theorem: asymptotics agree} we have that 
	\begin{equation*}
	\widehat{F} \left(\frac{h}{k} \right) = \underset{t \rightarrow 0^+}{\lim} F\left(e^{2 \pi i  \frac{sh}{k} - t}\right) = a_{hs_1, \frac{k}{s_2}} = \underset{t \rightarrow 0^+}{\lim} \mathbb{E} \left(- \frac{h}{k} + \frac{it}{2 \pi} \right) ,
	\end{equation*}
	where $s_1 \coloneqq s/ \gcd(s,k)$ and $s_2 \coloneqq \gcd(s, k)$.
	The claim then follows from Proposition \ref{Proposition: vector-valued QMF}. \qedhere
\end{proof}

\begin{example}\textit{(continued)}
	Returning again to our example, and using that $e(A), e(B)$ are constant across $\mathcal{A}, \mathcal{B}$ (as seen in Section \ref{Section: double Eichler integrals}), the above Theorem shows us that this example is a single component of a vector-valued quantum modular form of depth two and weight one with quantum set $\mathcal{Q} = \Q$ on $\text{SL}_2 (\Z)$.
\end{example}

\section{Completed Indefinite Theta Functions}\label{Section: Completed Indefinite theta functions 2}
We now view the function $\mathbb{E} (\tau)$ as the ``purely non-holomorphic" part of an indefinite theta series. For $A \in M_m (\Z)$ a non-singular $m \times m$ matrix, $P \colon \R^m \rightarrow \mathbb{C}$, and $a \in \mathbb{Q}^m$ we define the theta function

\begin{equation*}
\Theta_{A, P, a} (\tau) \coloneqq \sum_{n \in a + \Z^m} P(n) q^{\frac{1}{2} n^T A n} .
\end{equation*}

Set $A_1 \coloneqq \left( \begin{smallmatrix} 2a_1 & a_2 & 2a_1 & a_2 \\ a_2 & 2a_3 & a_2 & 2a_3 \\ 2a_1 & a_2 & 0 & 0 \\ a_2 & 2a_3 & 0 & 0 \end{smallmatrix} \right)$ with associated bilinear form $B_1(x,y) = x^T A_1 y$ and quadratic form $Q_1(x) \coloneqq \frac{1}{2} B_1(x,x)$. We also set $A_0 \coloneqq \left( \begin{smallmatrix} 2a_1 & a_2 \\ a_2 & 2a_3 \end{smallmatrix}
\right) $ and the function $P_0 (n) \coloneqq M_2 (\kappa; 2 \sqrt{a_1} n_1 + \frac{a_2}{\sqrt{a_1}} n_2 , mn_2)$, and for $n \in \R^4$ put 

\begin{equation*}
\begin{split}
P(n) \coloneqq & M_2 \left( \kappa; 2 \sqrt{a_1} n_3 + \frac{a_2}{\sqrt{a_1}} n_4 , mn_4 \right) + (\sgn(n_1) + \sgn(n_3) ) M \left( \frac{a_2 n_3 + 2 a_3 n_4}{\sqrt{a_3}} \right) \\
& + (\sgn (n_2) + \sgn (n_4)) M \left(\frac{2 a_1 n_3 + a_2 n_4}{\sqrt{a_1}} \right) \\
& +  (\sgn (2a_1 n_3 + a_2 n_4) + \sgn(n_1) ) (\sgn(a_2 n_3 + 2a_3 n_4) + \sgn (n_2) ).
\end{split}
\end{equation*}
Note that for $\alpha \in \J^*$ we have

\begin{equation*}
2 \mathcal{E}_{\alpha} (\tau) = \Theta_{- A_0, P_0, \alpha} (\tau).
\end{equation*}
In particular, we have the following Proposition.

\begin{proposition}\label{Proposition: indefinite theta function}
	The function $\mathcal{E}_{\alpha} $ can be viewed as the ``purely non-holomorphic part" of 
	\begin{equation*}
	\Theta_{A_1, P, a} (\tau) = \sum_{n \in a + \Z^4} P(\sqrt{v} n) q^{ Q_1(n)} ,
	\end{equation*}
	where $a \in \frac{1}{s} A_1^{-1} \Z^4$ with $(a_3, a_4) = (\alpha_1, \alpha_2)$. Moreover, $\Theta_{A_1, P, a} (\tau)$ is an indefinite theta function of signature $(2,2)$.
\end{proposition}

\begin{proof}
	To see the first claim, we put $P^- \coloneqq   M_2 (\kappa; 2 \sqrt{a_1} n_3 + \frac{a_2}{\sqrt{a_1}} n_4 , mn_4)$ and we then have
	
	\begin{equation*}
	\Theta_{A_1, P^-, a} (\tau) = 2 \mathcal{E}_{\alpha} (\tau) \Theta_{A_0, 1, (a_1 - a_3, a_2 - a_4)} (\tau) .
	\end{equation*}

	The authors in \cite{Higher_depth_QMFs} give a framework to proceed directly to prove the convergence and modularity properties of $\Theta_{A_1, P, a} (\tau)$ via a Theorem of Vign{\'e}ras, and so instead here we employ Section \ref{Section: boosted error functions} of the present paper, which follows from \cite{Generalised_error_functions}.
	
	Choosing $C_1 = (0,1,0,-1)^T$, $C_2 = (1,0,-1,0)^T$, $C_1' = (0,0,a_2,-2a_1)^T$, and $C_2' = (0,0,-2a_3,a_2)^T$ one can verify that the conditions in Section \ref{Section: boosted error functions} hold. Computing the completion given in Theorem \ref{Theorem: completion to non-holomorphic} along with those in the locally constant product of sign functions gives exactly the terms in $P(n)$ up to a factor of $\frac{1}{\sqrt{2}}$. Therefore, by Theorem \ref{Theorem: completion to non-holomorphic} we see that $\Theta_{A_1, P, a} (\tau)$ is a non-holomorphic theta series of weight $2$, clearly of signature $(2,2)$. \qedhere
\end{proof}

\begin{example}\textit{(continued)}
	For our example we set $A_1 \coloneqq \left( \begin{smallmatrix} 4 & 1 & 4 & 1 \\ 1 & 2 & 1 & 2 \\ 4 & 1 & 0 & 0 \\ 1 & 2 & 0 & 0 \end{smallmatrix} \right)$ along with $A_0 \coloneqq \left( \begin{smallmatrix} 4 & 1 \\ 1 & 2 \end{smallmatrix} \right) $. We also set
	
	\begin{equation*}
	\begin{split}
	P(n) \coloneqq & M_2 \left( \frac{1}{7}; 2 \sqrt{2} n_3 + \frac{1}{\sqrt{2}} n_4 , \frac{\sqrt{7}}{\sqrt{2}}n_4 \right) + (\sgn(n_1) + \sgn(n_3) ) M \left(  n_3 + 2 n_4 \right) \\
	& + (\sgn (n_2) + \sgn (n_4)) M \left(\frac{4 n_3 + n_4}{\sqrt{2}} \right) \\
	& +  (\sgn (4 n_3 + n_4) + \sgn(n_1) ) (\sgn( n_3 + 2 n_4) + \sgn (n_2) ).
	\end{split}
	\end{equation*}
	
	Then, for example, the function $\mathcal{E}_{\left(\frac{1}{4},\frac{1}{4} \right)}$ is seen to be the non-holomorphic part of 
	
	\begin{equation*}
	\Theta_{A_1, P, a} (\tau) = \sum_{n \in a + \Z^4} P(\sqrt{v} n) q^{ Q_1(n)} ,
	\end{equation*}
	with $a \in \frac{1}{4} A_1^{-1} \Z^4$ such that $(a_3, a_4) = \left(\frac{1}{4},\frac{1}{4} \right)$.
\end{example}

\section{Further questions}\label{Section: further questions}
We end by commenting on some related questions.

\begin{enumerate}
	\item Here we discussed only the case where we had a binary quadratic form. It is expected that a somewhat similar situation occurs for more general forms exists, though it is expected to be much more technical. Systematic study of higher depth analogues of $F$ is planned, however this will require more careful treatment due to a variety of factors. Work of Nazaroglu on higher dimension error functions \cite{Nazaroglu} gives a possible pathway to this. Discussions with Kaszian have also revealed some difficulty in defining the quantum set for higher dimensional versions of $F$ in general. In particular, in the present paper we exploited the fact that
	
	\begin{equation*}
	B_m (1-x) = (-1)^m B(x)
	\end{equation*}
	when determining the asymptotic behaviour of $F$. We were then able to deduce vanishing results for possible growing terms via the use of the Euler-Maclaurin summation formula in two dimensions. In higher dimensions this relationship becomes more complex, and using the $n$-dimensional Euler-Maclaurin formula (see e.g. Equation 3.1. in \cite{PeakPositions}) we find more growing terms that we require to vanish. In general, it is anticipated that this will place many restrictions on the quantum set.
	\item Further generalisation of the situation will be explored, in particular to include the holomorphic part of certain indefinite theta functions. That is, replace $Q(n)$ by an indefinite quadratic form of signature $(1,1)$, and take the holomorphic part of a certain family of indefinite theta functions following Zwegers' construction in \cite{zwegers2008mock}. Generically these look like
	
	\begin{equation*}
	\sum_{n \in a + \Z^r} \left( \sgn(B(c_1,n)) - \sgn(B(c_2,n)) \right) e^{2\pi i B(n,b)} q^{Q(n)},
	\end{equation*}
	where $a,b,c_1,c_2$ lie in $\Z^r$ and satisfy certain conditions to ensure convergence (see \cite{zwegers2008mock} for a full definition). Our question is then: are there certain families of these indefinite theta functions with higher depth quantum modularity?
	
	\item It is anticipated that in a following paper certain examples of the quantum modular forms here are viewed as $q$-series, using Larson's ``translation" \cite{larson2015generalized} of work by Griffin, Ono, and Warnaar on generalised Andrews-Gordon identities \cite{griffin2016}.
	\item In this paper we only realised the double Eichler integral (and therefore $F$) as components of a vector-valued quantum modular form, but we did not compute the other components of the vector-valued form explicitly. Further work in this direction is expected, and its implications in representation theory explored. 
\end{enumerate}

\bibliographystyle{amsplain}
\bibliography{Bibliography}

\end{document}